\documentclass[
reqno, 
12pt]{amsart}

\usepackage{amssymb,amsmath, amsthm,latexsym}
\usepackage{a4wide}

\usepackage{hyperref}
\usepackage{color}

\usepackage{tikz}

\theoremstyle{plain}

\newtheorem{lemma}{Lemma}[section]
\newtheorem{theorem}[lemma]{Theorem}
\newtheorem{proposition}[lemma]{Proposition}

\newtheorem{remark}[lemma]{Remark}

\newtheorem{definition}[lemma]{Definition}

\newtheorem{example}[lemma]{Example}

\parskip=\bigskipamount 

\font\rm=cmr12

\def\Z{\mathbb Z}

\def\d{\delta}

\def\t{\times}
\def\o{\otimes}

\def\h{\hookrightarrow}

\def\a{\alpha}

\def\D{\Delta}

\def\s{\sigma}

\title[Interactions between harmonic analysis and  automorphic forms]
{On interactions between harmonic analysis and the theory of automorphic forms}

\mark{Marko Tadi\'c}

\author{Marko Tadi\'c}

\address{Department of Mathematics, University of Zagreb
\\
Bijeni\v{c}ka 30, 10000 Zagreb,
 Croatia\\
Email: \tt tadic{\char'100}math.hr}

\keywords{non-archimedean local fields, classical groups, irreducible representations, unitary representations, square integrable representations, cuspidal representations, automorphic representations}
\subjclass[2000]
{Primary: 22E50, 22E55, Secondary: 11F70, 11S37}

\thanks{
The   
author was partly supported by 
Croatian Ministry of Science, Education and Sports grant
{\#}037-0372794-2804.}
\date{\today}

\begin{document}
  
\begin{abstract}
In this paper we review some connections between harmonic analysis and the modern theory of automorphic forms.  We indicate in some examples how the study of problems of harmonic analysis brings us to the important objects of the theory of automorphic forms, and conversely. We consider classical groups and their unitary, tempered, automorphic and unramified duals. The most important representations in our paper are the isolated points in these duals.
\end{abstract}

\maketitle

\setcounter{tocdepth}{1}

%\begin{centerline}
%{--------- \ \ \ \  Preliminary version \ \ \ \  ---------}
%\end{centerline}

\tableofcontents

\section{Introduction}\label{s-intro}

We  start by recalling  the very well known 
  concept, due to Gelfand, of harmonic analysis on a locally compact group. The main problem of harmonic analysis on a  group $G$ is to understand some important unitary representations of $G$, such as  $L^2(G)$. One approach to this problem is to break it into  two parts:

\begin{enumerate}

\item[(G1)] Describe conveniently (possibly fully classify) the unitary dual $\hat G$ of $G$, i.e., the set of all equivalence classes of irreducible unitary representations of $G$.

\item[(G2)] Decompose  important unitary representations of $G$  in terms of $\hat G$ (as direct integrals, for example).

\end{enumerate}

Observe that the main  importance of (G1) comes from (G2).

Some very important problems of the modern theory of automorphic forms are
typical problems of non-commutative harmonic analysis (in a broad sense), and progress in harmonic analysis has consequences in the theory of automorphic forms.

From the other side, in building  harmonic analysis on reductive groups,
automorphic forms are very useful. They are a very rich source of relevant
ideas and concepts.

We  review some of these connections in this  paper, giving the picture from our point of view, which is closely related to  harmonic analysis. We  review only connections that have appeared primarily 
related to 
 our work on problems of harmonic analysis. Therefore, a number of other very interesting connections  are omitted (among others, the work of F. Shahidi  contains very nice interactions).  In this paper, we concentrate on  relatively simple, but still interesting cases. Since both fields are very technical and often have  very  complicated notation, we  try to keep the technical part as simple as possible and   often give   suggestive examples rather than 
 the full results. We  often simplify the situation as much as possible (trying not to oversimplify). 

We  deal only  with  classical groups  in this paper. These groups have been  generators of progress at some crucial stages in the  development of  harmonic analysis 
as well as the modern theory of automorphic forms. 
We  now very briefly review the topics that we cover in the paper (one can find more details  in the paper).

The unitary dual $\hat G$ carries a natural topology  (defined in terms of approximation of matrix coefficients on compact subsets; see section 2.). For the commutative groups, the Pontryagin dual is formulated in this topology. In the non-commutative case, $\hat G$ does not need to be topologically homogenous. Of particular interest are the most singular representations of $\hat G$, the isolated points of this space, i.e.,  the isolated representations. In the case of unitary duals of reductive groups over local fields 
which are classified, usually isolated representations (of the group and its Levi subgroups)  generate the whole unitary dual using  some very standard 
constructions.

The part of the unitary dual of $G$ which takes part in the decomposition of a unitary representation $\Pi$ is called the support of $\Pi$, and it is denoted by
$$
\text{supp}(\Pi).
$$
Depending on the $\Pi$ which one considers, $\text{supp}(\Pi)$ can be a relatively small part of $\hat G$. Therefore, we may be able to avoid general problem (G1), at least for dealing with such unitary representations (and focus our attention only on the classification of representations in  $\text{supp}(\Pi)$).
 This was the case in  Harish-Chandra's fundamental work on the explicit decomposition of the regular representation of semi-simple real group $G$ on $L^2(G)$ (at the time when Harish-Chandra decomposed those representations for general semi-simple groups, unitary duals of simple groups were classified only for very low real ranks).  
 There are also  examples of unitary representations of different type. We    next consider the unitary representations of general linear groups which have  the most delicate parts of the unitary duals in their supports, all the isolated representations.
 
 The support of $L^2(G)$, which was classified by Harish-Chandra, is called the tempered dual of $G$, and it is denoted by
 $$
 \hat G_{\text{temp}}.
 $$

The support of a unitary representation $\Pi$ can have its own isolated representations (which do not need to be isolated in $\hat G$), and the fact that they are isolated here may indicate their relevance for the unitary representation that one studies. The following  fact is a nice example of this. For $\pi \in \hat G_{\text{temp}}$, we have
$$
\text{$\pi$ is isolated in  $ \hat G_{\text{temp}} \iff \pi$ is square integrable}.
$$

Let $G$ be a  reductive group defined over a number field $k$.
Besides the regular representation  (and the isolated points in its support), 
we also consider  the representation of a local factor $G(k_v)$ of the adelic group $G(\mathbb A_k)$ on the space of square integrable automorphic  forms $L^2(G(k)\backslash G(\mathbb A_k))$. The support of this representation is called the automorphic dual of $G$ at $v$  (see \cite{Cl-Park} or the fourth section for a precise definition). Automorphic duals are very important, since they are related to a number of very hard questions in number theory (starting with Selberg $\frac14$-conjecture). We do not know them even in the simplest cases, like $SL(2)$ or $GL(2)$. Nevertheless, we know pretty well what we can expect, at least in some cases. Further, we can prove a number of not-trivial facts about them. Representations in the automorphic dual which are unramified\footnote{The term spherical is also often used instead of unramified.} with respect to a fixed maximal compact subgroup of $G(k_v)$ (i.e., containing a non-trivial vector fixed by the maximal compact subgroup) may be of particular interest. 
They are called Ramanujan duals\footnote{The unramified classes in the unitary dual will be called the unramified unitary dual.} (following \cite{BLS}). We   also denote $k_v$ by $F$. In what follows, we take $F$ to be non-archimedean\footnote{In this case, we consider maximal compact  subgroups $G(\mathcal O_F)$, where $\mathcal O_F$ is the maximal compact subring of $F$.} (although a number of facts that we  discuss hold or are expected to hold also  in the archimedean case).

 Now we move to the case of general linear groups,  which are  better understood then the other reductive groups. Take an irreducible square integrable representation $\sigma$ of $GL(n,F)$. Equivalently, we might say that we took an isolated representation\footnote{In what follows, by isolated representation we shall always mean isolated modulo center (see the third section for the definition).} from the tempered dual of $GL(n,F)$. Fix a positive integer $m$ and consider the  representation
\begin{equation}
\label{eq-ind-1}
\text{Ind}_P^{GL(mn,F)}(|\det|_F^{(m-1)/2}\sigma\otimes |\det|_F^{(m-1)/2-1}\sigma\otimes\dots\otimes |\det|_F^{-(m-1)/2}\sigma),
\end{equation}
parabolically induced from the appropriate parabolic subgroup which is standard with respect to the minimal parabolic subgroup consisting of upper triangular matrices. The above representation 
has a unique irreducible quotient which is denoted by
$$
u(\sigma,m),
$$
and called a Speh representation. Then each  isolated representation in the unitary dual of a general linear group is a Speh representation, and Speh representations are almost always isolated ($u(\sigma,m)$ is isolated if and only if $m\ne 2$ and if $\sigma$ is  isolated\footnote{Equivalently,  $\sigma$ does not correspond to  a segment of cuspidal representations  of length two in the Bernstein-Zelevinsky theory.} in $\widehat{GL(n,F)}$). It is an important fact  that Speh representations are in the automorphic dual (\cite{Jq}). A further very important fact is that they are always isolated in the automorphic dual  (\cite{MulSp} and \cite{LoRuSa}). Moreover, each isolated representation in the automorphic dual is expected to be some Speh representation (this would hold if we assume the generalized Ramanujan conjecture).

We have seen that the condition of being isolated in the tempered dual has a precise  (and important)  representation theoretic meaning.  The precise (arithmetic) meaning of the condition of being isolated in the automorphic dual  is  less clear. Let us recall  a very important and elegant paper \cite{Ka} of D. Kazhdan, where he proves that the trivial representation is isolated in the unitary dual of a simple algebraic group of rank $\ne1$  over local field, and from this he derives some important arithmetic consequences.

Related to the result of D. Kazhdan, it is interesting to note that for $SL(n,F)$, $n\ne2 $, the trivial representation is also the only isolated representation in the unramified unitary dual for these groups.
Clearly, the trivial representation is automorphic and isolated  in the Ramanujan dual (also for $n=2$; this follows from a general fact proved in  \cite{Cl-tau}). Further, it is also  expected to be  the only isolated representation there. Therefore, the set of isolated representations in the unramified unitary dual and the  isolated representations in the Ramanujan dual are expected to coincide for $SL(n)$, with the exception $n=2$. We have a surprisingly different situation for
 other classical groups, as we  see next. We  consider the example of $Sp(340,F)$.

 The first surprise is that the
number of isolated representations in the unramified unitary dual of
$Sp(340,F)$ is
$$
 11\ 322\ 187\ 942
$$
(\cite{MuT} and \cite{T-auto}). Further, G. Mui\'c has proved an important fact that  these representations are all automorphic (\cite{Mu-unit}; this is also how he proved their unitarity). In this way we get a huge number of isolated representations in the Ramanujan dual (isolated points there consist of at least $11\ 322\ 187\ 942$ above  representations). 
The following surprise is that \cite{Mu-unit}, \cite{MuT} and Conjecture ``Arthur + $\epsilon$" of L. Clozel from \cite{Cl-Park}, would imply that the number of isolated representations in the Ramanujan dual is
$$
568\ 385\ 730\ 874,
$$
which is substantially bigger number than the number of isolated representations in the unramified unitary dual\footnote{We have an intrinsic characterisation  of the tempered dual and its isolated points among all the irreducible unitary representations (thanks to Harish-Chandra and W. Casselman). A similar situation might be the case regarding the Ramanujan dual and its isolated points (see the fifth section).}.

All this shows that compared to the $SL(n)$-case, new phenomena happen here (or are expected to happen).  A new fact is that we have a huge number of isolated representations  in the unramified unitary dual, and also in the Ramanujan dual. A new phenomenon  here is that  the  isolated unramified representations in the unitary dual are expected to form  a very small portion of  isolated representations in the Ramanujan dual.

The above example also  raises some questions.
The first one  is of an arithmetic nature.
Having in mind D. Kazhdan's paper \cite{Ka},  one may ask if  the above stunning difference regarding the number of isolated points for $SL(n)$ and $Sp(2n)$ groups has some arithmetic explanation or consequence?

The second question  is related to  harmonic analysis: why does  such a  small portion of  representations which are expected to be isolated in the Ramanujan dual remain isolated in the whole unramified dual (this was not the case for $SL(n)$)? The reason is that for each of $568\ 385\ 730\ 874$ strongly negative unramified representations, excluding $11\ 322\ 187\ 942$ of them (i.e., the isolated ones in the unramified dual), 
 there is a complementary series at whose end  this representation lies (and complementary series are not expected to be in the Ramanujan dual). Complementary series representations start with an irreducible representation parabolically induced from a unitary one. Therefore, in the example of $Sp(340,F)$, for the $557\ 063\ 542\ 932$ parabolically induced representations which are involved (where corresponding complementary series start),  we need to know their irreducibility. 
 
 The above short discussion indicates that  very often we have complementary series which end with representations that are expected to be isolated in the Ramanujan dual. But we have still a huge number of isolated representations in the unramified dual.
Therefore, the following  question arises: why are there still $11\ 322\ 187\ 942$   isolated representations in the unramified dual? The answer is roughly: no  complementary series ends with them (which is related to the question of reducibility of parabolically induced representations).

The above discussion suggests that if we want to know explicit answers regarding harmonic analysis or automorphic forms, we need also to have very explicit knowledge of complementary series (not just on some algorithmic level). This implies that to start, we need  to  have  a very explicit understanding of the question of irreducibility/reducibility of parabolically induced representations by  unitary ones  (unramified ones in this case).
Such an understanding of the required irreducibility/reducibility is obtained by G. Mui\'c in \cite{Mu-no-un}. We are not going to  explain it here, but rather we go to a different (and dual) setting, where such an understanding is also crucial, and explain how one can deal with the question of irreducibility/reducibility there.  
We shall see how these   basic questions of  harmonic analysis lead to some deep problems in the number theoretic setting. The majority of this paper is devoted to this case. Below, we  sketch only very basic idea.

Let us return to the fundamental  problem (G1).
A standard strategy for  (G1) is to classify the non-unitary dual of $G$ (formed of equivalence classes of irreducible representations of $G$), and then classify the unitarizable classes in it (i.e., $\hat G$).
The Langlands classification
of the non-unitary dual reduces this problem to the problem of tempered duals of its Levi subgroups. A very significant step in classifying  the tempered dual is classifying the irreducible square integrable representations.

We now  restrict to the case of classical groups (symplectic, orthogonal or unitary). For simplicity,  in this discussion we consider only  the case of symplectic groups (in the paper we also consider  the case of groups $SO(2n+1,F)$). Here the structure of Levi subgroups (which are direct  products of general linear groups and a symplectic group), and the existing classification of tempered duals of general linear groups, reduce the problem of the non-unitary dual to the problem of  classifying  the tempered duals of symplectic groups. To get irreducible tempered representations from the square integrable ones, one needs to classify all irreducible subrepresentations of the representations parabolically induced from the irreducible square integrable ones. We get the simplest example of such induced representations  if we  take  irreducible square integrable representations $\delta$ and $\pi$ of a general linear group and a symplectic group, and consider the  representation
\begin{equation}
\label{eq-ind-2}
\text{Ind}(\delta\otimes\pi)
\end{equation}
of a symplectic group,  parabolically induced from a maximal parabolic subgroup. Actually, the theory of $R$-groups reduces the general case to the question of whether the  representations \eqref{eq-ind-2} reduce (see \cite{Go}). The representation \eqref{eq-ind-2} can be reducible only if $\delta $ is self dual (i.e., equivalent to its own contragredient). Therefore, we  assume this  in what follows.

One   way to try to understand the reducibility of \eqref{eq-ind-2} is the following. Suppose $\sigma$ in the induced representation \eqref{eq-ind-1} is   an irreducible cuspidal representation   of a general linear group with unitary central character (cuspidal representations are characterized by the property that their matrix coefficients are compactly supported modulo center - they can be characterized as isolated representations of the non-unitary dual, which also carries a natural topology). Then the representation \eqref{eq-ind-1}  contains a unique irreducible subrepresentation, which we  denote by
$$
\delta(\sigma,m).
$$
This representation is  square integrable, and one gets all square integrable representations in this way (\cite{Z}). Now we can slightly reinterpret the question of reducibility of 
\eqref{eq-ind-2}. It is equivalent to
 the question of reducibility of 
\begin{equation}
\label{eq-ind-3}
\text{Ind}(\delta(\sigma,m)\otimes\pi),
\end{equation}
when $\sigma$ is a self dual irreducible cuspidal representation of a general linear group. If we fix $\sigma$ and $\pi$ as above, then there is one parity of positive integers (even or odd), such that for representations $\delta(\sigma,m)$, with $m$ from that parity, the representation \eqref{eq-ind-3} is  always irreducible (this parity  depends only on $\sigma$). For the  representations $\delta(\sigma,m)$ with $m$ from the other parity, the representation \eqref{eq-ind-3} is  always reducible, with finitely many exceptions. Denote  by
$$
\text{Jord}(\pi)
$$
the set of all such exceptions $\delta(\sigma,m)$ (for fixed $\pi$; we let  $\sigma$  run over all equivalence classes of self dual irreducible cuspidal representation of  general linear groups)\footnote{C. M\oe glin has defined Jordan blocks (slightly differently; one can find  her original definition  in \cite{Moe-Ex}).   Here we use a  different notation  than in  \cite{Moe-Ex} and the other papers, where elements of $\text{Jord}(\pi)$ are pairs $(\s,m)$ instead of square integrable representations $\d(\s,m)$  (recall  that  $(\s,m)\leftrightarrow \d(\s,m)$ is a bijection by \cite{Z}).}. In other words, roughly speaking  $\text{Jord}(\pi)$ takes  care of all the singularities of the parabolic induction of \eqref{eq-ind-3}. Therefore, it is a crucial object for understanding the tempered representations which can be obtained from $\pi$. We  illustrate the importance of $\text{Jord}(\pi$) with the following 

\begin{example}
\label{ex-int}  In this example, we consider odd orthogonal groups. A direct consequence of \cite{Sh2} and \cite{T-irr} is that 
$$
\text{Jord}(1_{SO(1,F)})=\emptyset.
$$
Thus, for $m$ from one parity (depending on $\sigma$), the representation 
$$
\text{Ind}(\delta(\sigma,m)\otimes1_{SO(1,F)})
$$
 is always reducible, while for $m$ from the other parity it is always irreducible ($\sigma$ is an irreducible self dual cuspidal representation of a general linear group). Now, for a self dual square integrable representation\footnote{Recall that $\text{Ind}(\delta\otimes\pi)$ is irreducible if $\delta$ is not self dual.} $\d$ we have
$$
\text{$\text{Ind}(\delta\otimes\pi)$ is reducible $\iff \text{Ind}(\delta\otimes1_{SO(1,F)})$ is reducible and $\d\not\in\text{Jord}(\pi)$}.
$$
In other words, roughly Jord$(\pi)$ measures the difference between  tempered induction of $\pi$ and the trivial representation $1_{SO(1,F)}$\footnote{For symplectic groups, the above discussion also holds if we exclude the trivial representation of $GL(1,F)$. This difference is caused by the fact that  $\text{Jord}(1_{SO(1,F)})=\{1_{GL(1,F)}\}$, which again directly follows from \cite{Sh2} and \cite{T-irr}.}.
This difference is not very big since Jord$(\pi)$ must be finite. For the Steinberg representation, we have
$$
\text{Jord}(\text{St}_{SO(2n+1,F)})=\{\delta(1_{F^\times},2n)\}=\{\text{St}_{GL(2n,F)}\}.
$$
\end{example}

It is interesting and very important  that $\text{Jord}(\pi)$ has  arithmetic meaning. We  describe this briefly.  
The Langlands program predicts a natural parameterization of an irreducible  representation $\tau$ of a split reductive groups $G$ over $F$ by a homomorphism 
$$
\Phi_{G}(\tau) 
$$
of $W_F\times SL(2,\mathbb C)$ into the complex dual Langlands  group, satisfying certain requirements (such homomorphisms are called admissible). The parameterization $\tau \mapsto  \Phi_{G}(\tau)$ is called the local Langlands correspondence for $G$ (these correspondences  are expected to be instances of a more general phenomenon, called functoriality). Such correspondences can  be viewed as generalizations of the local Artin reciprocity law from  class field theory. Representations with the same parameter $
\Phi_{G}(\tau)
$ are called $L$-packets (these sets are expected to be finite). The representations inside $L$-packets are expected to be parameterized by equivalence classes of  irreducible representations of the component group of $\Phi_{G}(\tau)$ (see the sixth  section for more details).
More then a decade ago, the existence of such a correspondence was established for general linear groups in full generality (\cite{LaRpSt}, \cite{HT} and \cite{He}). We denote them by $\Phi_{GL}$ (here $L$-packets are singletons).

 One of the  very  big breakthroughs in the theory of automorphic forms,  obtained by J. Arthur  in his  recent book \cite{A-book}, is a classification of irreducible tempered representations of classical $p$-adic groups\footnote{Arthur classification is still conditional, but this is expected to be removed  soon (when the facts on which  \cite{A-book} relies become available). Therefore, we shall use  \cite{A-book} in what follows  without mentioning that there is still a piece to be completed.}. This classification can be viewed as an instance of a local Langlands correspondences. In the case of unitary groups, such a classification was obtained earlier by C. M\oe glin in \cite{Moe-Pac}.  A fundamental result, which  tells us that crucial objects of  harmonic analysis are directly related to the fundamental objects of the number theory, is  the following theorem of C. M\oe glin:
 
\begin{theorem}
\label{th-Moe-int}
 The admissible homomorphism that J. Arthur has attached to a square integrable representation  $\pi$ is 
\begin{equation}
\label{int-eq-sum-L}
\underset{\s\in \text{Jord}(\pi)}
\oplus
\Phi_{GL}(\sigma).
\end{equation} 
\end{theorem}

Therefore, 
Arthur's classification singles out    crucial information 
for the tempered induction 
related to $\pi$. 

Besides this information, there are a number of other questions important for  harmonic analysis which  still remain to be answered. One of them is how irreducible square integrable representations are built from   cuspidal ones.  In the case of general linear groups, this question is answered by the Bernstein-Zelevinsky theory. For other classical groups, this question is directly related to  understanding the internal structure of packets\footnote{They are not called $L$-packets, since \cite{A-book} does not address the question of $L$-functions.}.

C. M\oe glin has characterized  the  parameters  corresponding to the  cuspidal representations in the Arthur classification. They have a very  simple description (see the seventh section). Now we can use the classification of irreducible square integrable representations of classical $p$-adic groups modulo cuspidal data (obtained in \cite{Moe-Ex} and \cite{Moe-T}), to obtain a description of irreducible square integrable representations of classical $p$-adic groups in terms of   cuspidal ones. This also gives representation theoretic information on how packets are build.

There are a number of questions that one can further consider regarding the structure of the packets. We  give one  example. A new feature showing up in the Arthur classification (which was not present for groups like $GL(n)$ or $SL(n)$), is the existence of square integrable packets containing both cuspidal and non-cuspidal representations at the same time. The extreme instance of this phenomenon is (square integrable) packets containing at the same time  a representation supported on the minimal parabolic subgroup and a cuspidal representation. Such packets are called packets with antipodes. A simple application of the representation theoretic description of packets is the fact that  
$SO(2n+1,F)$ has   a packet of this type  if and only if $n$ is even.

We are very thankful to C. M\oe glin for discussions and for providing us with some references 
in 
the last  sections of the paper. 
G. Savin  has read the first version of this paper, and  gave us a number
of useful suggestions. I. Mati\'c  also gave suggestions to the following version of the paper. C. Jantzen's numerous suggestions helped us a lot to  improve the style of the paper. I. Badulescu gave us a number of useful mathematical remarks to improve the paper. Discussions with M. Hanzer, A. Moy and G. Mui\'c were helpful during the preparation of this paper. The referee's numerous corrections and general remarks helped a lot to improve the exposition of the paper. We are  very thankful to all of them.

We now briefly discuss the contents of the paper section by section. The second section reviews very basic notions related to the natural topology on representations. 
In the third section, we
consider the example of $GL(n)$ and the isolated representations in this case.
 The fourth section introduces the automorphic duals, and follows the case of $GL(n)$. The fifth section studies the unramified duals of classical groups and automorphicity in this setting. In the sixth section, we follow how questions of harmonic analysis bring us to the   square integrable packets of classical groups, and the recent work of J. Arthur and C. M\oe glin. The seventh section recalls M\oe glin's description of cuspidal representations in the Arthur classification, which is followed up in the eighth section by a description of  the internal structure of packets.

\section{Topology, isolated representations, support}
\label{s-top}
\setcounter{equation}{0}
\setcounter{footnote}{0}

The unitary dual $\hat G$ is  a topological space in a natural way: $\pi$ is
in the closure of $X \subseteq \hat G$ if  and only if diagonal matrix
coefficients of $\pi$ on compact subsets can be approximated by 
finite sums of diagonal matrix coefficients from $X$. One can find more details in \cite{Di} and \cite{Fe}, or \cite {Mi} (the non-archimedean case is in \cite{T-Geo}).

 If $G$ is commutative then the Pontryagin duality is formulated in this topology. In this case $\hat G$ is a group and hence topologically homogenous.
 In the non-commutative case $\hat G$ does not need to be homogenous. Of particular importance are isolated representations (or isolated modulo center\footnote{If $G$ has non-compact center, the role of isolated representations is played by isolated representations modulo center. Let $\omega_\pi$ be the central character of $\pi \in \hat G$ and denote $\hat G_{\omega_\pi}=\{\tau\in\hat G; \omega_\tau=\omega_\pi\}$. Then we say that $\pi$ is isolated modulo center if $\pi$ is an isolated point of the topological space $\hat G_{\omega_\pi}$. These representations we  often simply call  isolated representations.}). They  are pretty  mysterious, but very important objects. Their unitarity is usually a very non-trivial fact. Local components of square integrable automorphic forms are a big source of isolated (or isolated modulo center) representations.

For performing step (G2) of the Gelfand concept on a fixed unitary representation $\Pi$ of $G$, we usually do not need the whole $\hat G$, but only the representations which are in the support of the measure on $\hat G$ which decomposes $\Pi$ into a direct integral of elements of $\hat G$ (we shall not go  into detail here regarding direct integrals).  
This support of the measure is denoted by
$$
\text{supp}(\Pi).
$$
One can describe the support of $\Pi$ without finding  the measure, but 
using only the topology. For $\pi\in\hat G$, we have: $\pi\in \text{supp}(\Pi)$ if and only if it is
 weakly contained in   $\Pi$, i.e., 
 if  diagonal
matrix coefficients of $\pi$ on compact subsets can be approximated by
finite sums of diagonal matrix coefficients of $\Pi$
(see \cite{Di} or \cite{Fe} for more details).

Clearly, $\text{supp}(\Pi)$ inherits the topology from $\hat G$. Even for $\text{supp}(\Pi)$,  isolated representations can be again very distinguished.
Let us illustrate this with the following:

\begin{example}
\label{ex-L2}
{\rm Let $G$ be a semi-simple algebraic group over a local field $F$. Then the support of the regular representation of $G$ on $L^2(G)$ by  right translations is
$$
\text{supp}(L^2(G))=\hat G_{\text{temp}},
$$
where 
$\hat G_{\text{temp}}$ denotes the set of all tempered representations in $\hat G$, i.e., those  whose matrix coefficients are in $L^{2+\epsilon}(G)$ for each $\epsilon >0$. A very important class of tempered representations (for harmonic analysis, as well as for the theory of automorphic forms) are square integrable representations (i.e., those ones whose matrix coefficients are in $L^{2}(G)$). Now for $\pi\in \hat G_{\text{temp}}$, we have 
$$
\text{$\pi$ is square integrable $\iff$ $\pi$ is isolated in $\hat G_{\text{temp}}$}.
$$
There is a very effective criterion  of Harish-Chandra and of W. Casselman, for checking square integrability of an irreducible representation. }

\end{example}

Let us recall that the essential part of the  monumental work of Harish-Chandra \cite{HC} was related  to the regular representation of a semi-simple group $G$. In the real case, Harish-Chandra has constructed all the isolated points of $\hat G_{\text{temp}}$'s. This  was enough for him to explicitly decompose  $L^2(G)$. The tempered dual $\hat G_{\text{temp}}$  was classified by A. Knapp and G. Zuckermann later in \cite{KnZu}, based on the work of Harish-Chandra. Here isolated representations in the tempered duals were crucial for the construction of the whole tempered dual (we shall see similar examples later).

In the book \cite{GN} of I.M. Gelfand and M.A. Naimark on harmonic analysis on complex classical groups, the authors wrote the  lists of irreducible unitary representations of those groups. They expected that the representations from  the lists  form  unitary duals of complex classical groups  (their lists were very simple). In the case of symplectic and orthogonal groups, the incompleteness of the lists
 was  clear pretty soon.
For the special linear groups, E.M. Stein  constructed in \cite{St}, in a relatively simple way, representations (complementary series) which were not in the lists of Gelfand and Naimark (for $SL(2n,\mathbb C)$, $n\geq 2$).

The above simple construction of E.M. Stein, the lack of significant progress in giving an  explicit classification of the whole unitary dual for a long time (even for the groups like $SL(n,\mathbb C)$) and very complicated approaches to the problem, have sometimes resulted  in  doubts that the unitary dual is the right object for  harmonic analysis, that it may be too big and complicated, consisting perhaps mainly of non-relevant representations for  harmonic analysis. 
That very successful   strategy of Harish-Chandra might be the right way for the general approach: to go directly to (G2)  for a specific important unitary representation (bypassing the general problem (G1)), and concentrate only on the part of $\hat G$ which is relevant for the unitary representation that we consider.

Considering some important groups, we shall see below   why this strategy is not likely to be very successful for the general case. Already some important unitary representations that show up in the theory of automorphic forms indicate this. For example, we shall see for some groups  that the most delicate part of the unitary duals, the isolated representations in $\hat G$, all show up. These representations are very  distinguished representations, and important for number of other problems.

In the unitary duals
 appear  very big and very complicated families of non-isolated representations (complementary series), which are not expected to show up in the  automorphic setting, at least not in the split  case. Therefore, from the point of view of automorphic forms, unitary duals may look too big, with significant parts which do not seem relevant. We shall see that even this part of the unitary duals  can be  interesting for the theory of automorphic forms.

To get an idea of isolated representations, their relation to automorphic forms, and the role of complementary series, we  go to the relatively  well understood case of $GL(n)$:

\section{The example of $GL(n)$}
\label{s-gl}
\setcounter{equation}{0}

All  parabolic subgroups of the groups that we  consider in this paper, are assumed to be standard with respect to the minimal parabolic subgroup consisting of upper triangular matrices in the group.

Let $F$ be a local field (or the ring of adeles of a global field). Let $\s$ be an irreducible square integrable  representation\footnote{These representations  are also  called   square integrable representations modulo center, since the requirement is that the absolute value of their matrix coefficients be square integrable functions modulo center.} of $GL(n,F)$ (in the adelic case we take an irreducible cuspidal representation of adelic $GL(n)$). 
Fix a positive integer $m$. Let $P$ be the  parabolic subgroup of $GL(nm,F)$ whose Levi subgroup is in a natural way isomorphic to 
$$
GL(n,F)\t \dots \t GL(n,F).
$$
  Consider the parabolically induced representation
\begin{equation}
\label{eq-ind}
\text{Ind}_P^{GL(mn,F)}(|\det|_F^{(m-1)/2}\s\o |\det|_F^{(m-1)/2-1}\s\o\dots\o |\det|_F^{-(m-1)/2}\s)
\end{equation}
($|\ |_F$ denotes the normalized absolute value on $F$). The above representation 
has a unique irreducible quotient, which is denoted by
$$
u(\s,m),
$$
and called a Speh representation. Observe that if $\s$ is a unitary character of $F^\times$, then $u(\s,m)=\s \circ \text{det}$ is also a character of $GL(m,F)$. This is the reason that  for archimedean $F$, one gets Speh representations which are not characters only if $F=\mathbb R$ and $\s$ is a square integrable representation of $GL(2,\mathbb R)$ (this is where the name  comes from; see \cite{Sp}).

Speh representations are very important in the theory of automorphic forms. 
It is interesting that we first came (in 
\cite{T-AENS}) to the Speh representations and their unitarity without knowing their role in the automorphic forms, studying  complementary series (which  are not expected to show up 
in the setting of automorphic forms in this case; we  comment on this later). We  briefly sketch below how we came to the Speh representations.

Obviously $u(\s,1)=\s$ is unitary (square integrable representations are unitary). Suppose that $u(\s,m)$ is unitary. Consider the family (complementary series)
\begin{equation}
\label{eq-comp}
\text{Ind}_{P'}^{GL(2mn,F)}(|\det|_F^{\a}u(\s,m)\o |\det|_F^{-\a}u(\s,m)), \quad 0\leq \a<1/2,
\end{equation}
where $P'$ is the appropriate  parabolic subgroup. This is irreducible for $\a=0$ by \cite{Be-P-inv} (and also it is unitary). For other $\a$'s as above, it is also irreducible, which is easy to see. Further, these representations are Hermitian. From this it easily follows that all the representations in \eqref{eq-comp} are unitary. If we put in \eqref{eq-comp}  $\a=1/2$ (the induced representation is then no more irreducible),     we get for a subquotient  the  representation
\begin{equation}
\label{eq-end}
\text{Ind}_{P''}^{GL(2mn,F)}(u(\s,m+1)\o u(\s,m-1)),
\end{equation}
if we can prove that \eqref{eq-end} is irreducible. Then the representation \eqref{eq-end} is unitary , since it is at the end of the complementary series (this follows from \cite{Mi}). Now a simple construction of unitary representations, which we call unitary parabolic reduction (since it is opposite to the unitary parabolic induction), implies that $u(\s,m+1)\o u(\s,m-1)$ is unitary, and thus also $u(\s,m+1)$ is unitary.

Now we shall see that  Speh representations are distinguished from the point of view  of harmonic analysis.
We assume in the rest of the paper that $F$ is a local non-archimedean field (although some facts also hold  in the archimedean case). Now, let $\sigma$ be a unitary irreducible cuspidal representation\footnote{These representations are  very distinguished  square integrable representations. Their  matrix coefficients are compactly supported modulo center. Irreducible cuspidal representations can be  characterized as isolated points in the non-unitary dual (see \cite{T-Geo}).} of $GL(n,F)$. Then the representation \eqref{eq-ind}  has a unique irreducible subrepresentation, which will be denoted by
$$
\d(\s,m).
$$
This representation is square integrable, and J. Bernstein has shown that one gets all such representations in this way. An old result from \cite{T-top-GL} gives the following characterization of isolated points in the unitary dual of general linear groups:

\begin{theorem}
\label{th-iso-gl}
Let $\pi\in \widehat{GL(k,F)}$. Then $\pi$ is isolated (modulo center) if and only $\pi\cong u(\d(\rho,l),m)$ for some  irreducible unitary cuspidal representation $\rho$ of $GL(k/(lm),F)$ and some positive integers $l\ne 2$ and $m\ne 2$ (clearly then $lm$ divides $k$).
\end{theorem}

In other words, if we have an isolated representation, then it is always a Speh representation. In the converse direction, a Speh representation is almost always isolated. The condition $l\ne 2$ and $m\ne 2$    spoils    the picture  given by the theorem a little bit\footnote{We have similar ``irregularity" in \cite{Ka} (rank must be $\ne 1$). This ``irregularity" was removed in the automorphic dual by L. Clozel (see \cite{Cl-tau}).} (the corresponding representations are subquotients of ends of complementary series, which easily implies that they can not be isolated).
 We shall  soon see that we get a completely regular picture  in the automorphic dual, which will be discussed in the following section.

{\bf Classification}: 
\begin{enumerate}
\item
It is very easy to state the classification of $\widehat{GL(n,F)}$ modulo square integrable representations: each representation  parabolically induced by a tensor product of Speh representations and complementary series \eqref{eq-comp} with $\a>0$, is in $\widehat{GL(n,F)}$, and  each representation $\pi$ in $\widehat{GL(n,F)}$ is obtained in this way. Further,  $\pi$ determines the Speh representations and complementary series inducing to $\pi$, up to a permutation. The classification holds also in the archimedean case in  exactly  the same form  and there irreducible  square integrable representations are very simple (see \cite{T-AENS} and \cite{T-R-C-old} or \cite{T-R-C-new} for both cases)\footnote{The complex case is particularly simple: each irreducible unitary representation   is parabolically induced  by tensor products of characters and complementary series starting with characters (constructed by E.M. Stein). This list differs from the list of I.M. Gelfand and M.A. Naimark in \cite{GN} only in that they have omitted complementary series of $GL(2n,\mathbb C)$ for $n\geq 2$ (i.e., Steins's complementary series).
The book of Gelfand and Naimark was much ahead of its time when it was published (1950), in particular regarding the intuition. It had strong influence on a number of further developments in harmonic analysis, as well as out of it.}.

\bigskip
\item
The above result (covering the non-archimedean as well as archimedean case, with the proofs along the same strategy) is just ``the tip of the iceberg". It is only a special case of a much more general result for  any local (finite dimensional) central division algebra\footnote{Recall that in the non-archimedean case the Brauer group is $\mathbb Q/\Z$ (the field case corresponds to the neutral element).} $A$ over $F$. For an irreducible square integrable representation $\s$ of $GL(n,A)$, let $s(\s)$ be the minimal positive exponent such that $\text{Ind}_{P'''}^{GL(2n,A)}(|\det|_F^{s(\s)/2}
\s\o |\det|_F^{-s(\s)/2}\s)$ reduces ($s(\s)$ is an integer  dividing the rank of $A$).
Then,  thanks mainly to the recent work of I. Badulescu, D. Renard and V. S\'echerre, if we define now representations $u(\s,m)$  and complementary series for general linear groups over $A$ putting   $|\det|_F^{s(\s)}$ instead of $|\det|_F$  in \eqref{eq-ind} and \eqref{eq-comp}, the above classification holds in the same form\footnote{Here representations $u(\s,m)$ are usually not in the automorphic dual (see \cite{Bad-Inv} and \cite{BR-arch}), but their unitarity easily follows using closely related automorphic representations (see \cite{BR-Tad})}
 for all groups $GL(n,A)$.

\end{enumerate}

The above classification of $\widehat{GL(n,A)}$ is obtained along the same strategy in the non-archimedean and archimedean case (the outline of that strategy is in \cite{T-Bul-naj}). This strategy reduces the classification to proving the unitarity of Speh representations, and  the irreducibility of unitary parabolic induction. We have described one possibility to prove the unitarity of Speh representations using complementary series (this strategy was used in \cite{T-AENS}, \cite{BHLS} and \cite{Bad-Speh}). In the following section we  comment on the possibility of proving unitarity of Speh representations  using automorphic forms (used first in \cite{Sp}, and then in \cite{T-AENS} and \cite{BR-Tad}; this approach does not distinguish  between the archimedean and non-archimedean cases). 
In the field case, the idea outlined in \cite{Ki} was a basis of proofs of irreducibility of unitary parabolic induction in \cite{Be-P-inv} and \cite{Bu} (see also \cite{AG}, \cite{SZ} and \cite{AGRS}).

\begin{remark}
\label{vo-cl}
\begin{enumerate}
\item
In the (non-commutative) division algebra case, D. Vogan  proved in \cite{Vo} the irreducibility of unitary parabolic induction    for Hamiltonian quaternions  
 (see \cite{BR-arch} for explanation how to get this irreducibility from \cite{Vo}) and  V. S\'echerre proved this irreducibility  in \cite{Se} for the case of non-archimedean division algebras (here we have plenty of division algebras). The two above proofs for division algebras are completely different. 
 A uniform proof for both cases would be very desirable (the idea of Kirillov from \cite{Ki} does not work here, neither can the approach of \cite{AG}, \cite{SZ} and \cite{AGRS}  be used here). A uniform proof in the division algebra case  would  also shed a new light on the field case. We expect such proof to be of a functional analytic nature.

\item
To complete this discussion, let us mention that there is the classification of D. Vogan in the archimedean case. His classification is given by Theorem 6.18 of \cite{Vo}, and is completely different from the classification that we have presented above. Stating his classification here would require a number of technical and combinatorial notions (and results), starting with several kinds of $K$-types. This is the reason that we do not present his classification here.
 His classification is equivalent to the
 specialization to the archimedean case of the classification that we presented above\footnote{Vogan's classification gives the complete answer in the archimedean case, while the classification that we presented above (holding in both cases) only reduces unitary duals to square integrable representations. Despite this, the above classification in the archimedean case  directly yields  the complete classification, since the required square integrable representations were known already in 1950's (and even earlier). In the non-archimedean case, this  was a very hard problem (see \cite{Z}, \cite{BK}, \cite{LaRpSt}, \cite{HT} and \cite{He}).}. 
 This equivalence is a non-trivial fact, and it is proved in \cite{BR-arch}.   
 \end{enumerate}

\end{remark}

\section{The automorphic dual}
\label{s-auto}
\setcounter{equation}{0}
\setcounter{footnote}{0}

\begin{definition}
\label{def-auto} 
Let $G$ be a reductive group defined over a number field $k$, $v$ a place of $k$,  $k_v$ the completion of $k$ at $v$ (which we often denote  by $F$), and $\mathbb A_k$ the ring of adeles of $k$. Then the automorphic dual 
$$
\hat G_{v,\text{aut}}
$$
is defined to be the support of the representation of $k_v$-rational points $G(k_v)$ of $G$ on the space of square integrable automorphic forms
$$
L^2(G(k)\backslash G(\mathbb A_k)).
$$
If an appropriate maximal compact subgroup $K$ of  $G(k_v)$ is fixed, we  denote by $\hat G_{v,\text{aut}}^{\mathbf 1}$ the representations in $\hat G_{v,\text{aut}}$ which possess a non-trivial vector for $K$ (i.e., the unramified part\footnote{The superscript $^{\mathbf  1}$ on a set of representations will always denote the unramified representations in that set.}). This part will be called the Ramanujan dual (as in \cite{BLS}).
\end{definition}

One can find the original (more general) definition, and much more detail in \cite{Cl-Park} (see also \cite{BLS}).

The  automorphic duals are very hard. They are  related to such  hard problems as  Selberg's $\frac 14$-conjecture, and more generally Ramanujan's conjecture for Maass forms, generalized Ramanujan conjecture, etc.
We do not know the classification of automorphic duals even in the simplest cases.
  Despite this, we know rather interesting facts for classical groups, as we shall see, and have more precise expectations than in the case of unitary duals. This may be good for both the duals.

H. Jacquet has proved in \cite{Jq} that Speh representations are automorphic (being local factors of global Speh representations, for which he proved that they are in the residual spectrum of the representation on the space of square integrable forms).
B. Speh has proved this earlier for $u(\s,m)$'s when $\s$ is a square integrable representation of  $GL(2,\mathbb R)$. Jacquet's proof clearly implies the unitarity of (local) Speh representations. The division algebra case requires an additional step (using unitary parabolic reduction, see \cite{BR-Tad}).

In this way, thanks to Theorem \ref{th-iso-gl} and H. Jacquet's result that we mentioned above, we get  plenty of isolated (modulo center) representations in the automorphic dual (although this term was not yet formally defined when the both results were available).

Further, the estimate in \cite{MulSp} of B. Speh and W. M\"uller of  local components of irreducible representations  in the residual  spectrum (which is based on earlier estimates of local unramified components of irreducible representations  in the cuspidal spectrum in \cite{LoRuSa}) and Langlands' description of automorphic spectra (\cite{A-Cor}) imply  that the Speh representations excluded by Theorem \ref{th-iso-gl} are also isolated in the automorphic dual.

Actually, the generalized Ramanujan conjecture (stating that each local  component of an irreducible representation  in the cuspidal spectrum of $GL(n)$ should be tempered), would imply that the Speh representations are the only isolated representations in the automorphic dual.

Moreover, the generalized Ramanujan conjecture also tells us what the automorphic dual should be. Let us briefly comment on this.
In J. Bernstein's work on unitarity \cite{Be-P-inv}, a notion of rigid  representations of $GL(n,F)$ naturally arose.  We can define these representations
as those  for which the essentially tempered representation $|\det|_F^{\a_1}\tau_1\o \dots\o |\det|_F^{\a_k}\tau_k$ (of a Levi subgroup) corresponding to it by the Langlands classification of the non-unitary duals ($\tau_i$ are tempered and $\a_i\in\mathbb R$) have all the exponents  $\a_i$ in $(1/2)\mathbb Z$. Denote by $\widehat{GL(n,F)}_\text{rig}$ the subset of rigid representations in $\widehat{GL(n,F)}.$ Then the classification theorem (\cite{T-AENS}, \cite{T-R-C-new}, or see our previous  brief description of the classification) implies that $\widehat{GL(n,F)}_\text{rig}$ consists of representations parabolically induced by tensor products of Speh representations (no complementary series). A. Venkatesh denotes $\widehat{GL(n,F)}_\text{rig}$ by $\widehat{GL(n,F)}_\text{Ar}$, since these representations come from the work of J. Arthur (\cite{A-unip}). 

A. Venkatesh observed in \cite{Ve} that the generalized Ramanujan conjecture would imply equality of the automorphic and the rigid duals, and conversely.
Namely, Langlands' description of automorphic spectra implies that $\widehat{GL(n,F)}_\text{rig}$ is contained in the automorphic dual, and further,
the generalized Ramanujan conjecture would imply equality. In the other direction, knowledge of equality at all places  would imply (using \cite{Shal}) the generalized Ramanujan conjecture.

One can find plenty of interesting facts, estimates and conjectures in \cite{Cl-Park} and \cite{Sa} (and papers  cited there).

\section{Unramified duals and Ramanujan duals of classical groups}
\label{s-unr}
\setcounter{equation}{0}
\setcounter{footnote}{0}

In this section, we  consider unramified irreducible unitary representations ($F$ is a non-archemedean field). For this section, it is  convenient to introduce  notion of  negative and  strongly negative representations. These terms were introduced by G. Mui\'c in \cite{Mu-no-un}. We have already mentioned the Casselman square integrability  criterion  (\cite{C-int}). His criterion for square integrability (resp., temperedness) modulo center, is given by some inequalities $<$ (resp.,  $\leq$). Reversing these  inequalities, one gets the definition of strongly negative (resp., negative) representations.
We do not go  into further detail here (see \cite{Mu-no-un} or Definition 1.1 in \cite{T-auto}). Irreducible negative (resp., strongly negative) representations are dual to the tempered (resp., square integrable) irreducible representations by duality of A.-M. Aubert  (\cite{Au}) and  P. Schneider and U. Stuhler (\cite{ScSt}). Unramified irreducible  negative representations are always unitary by \cite{Mu-unit} (we  comment on this later).
The simplest example of a strongly negative representation of $G$ is the trivial representation $1_G$.

We  first consider the simple example of $SL(n)$:

\begin{example}
\label{ex-sl}
Below, $k$ is an algebraic number field, $v$ a place of $k$ and $F=k_v$, the completion of $k$ at $v$. We have
\begin{enumerate}

\item 
$
\text{Isolated representations in }
\widehat{SL(n,F)}^{\mathbf 1}=
\begin{cases}
\{1_{SL(n,F)}\} & n\ne2,
\\
\emptyset & n=2.
\end{cases}
$

\item 
$
\text{Strongly negative representations in }
\widehat{SL(n,F)}^{\mathbf 1}=
\{
1_{SL(n,F)} 
\}.
$

\item
$
\text{Known (to us) isolated representations in }
\widehat{SL(n)}^{\mathbf 1}_{v,\text{aut}}=
\{
1_{SL(n,F)} 
\}.
$ 
\end{enumerate}

\end{example}

Therefore, the above three sets coincide, with one exception (the case $n=2$ in (1)).

Now we  discuss the case of other classical groups. To simplify exposition, we  concentrate on symplectic groups. Here we  use a classification of unramified unitary duals obtained in \cite{MuT}. 
By \cite{MuT}, an unramified irreducible unitary  representation is either negative, or a complementary series starting with a negative irreducible representation. Clearly, complementary series can not be isolated. Therefore, we are left with irreducible negative representations. Further by \cite{Mu-no-un}, each  negative representation is a subrepresentation of a representation parabolically induced by a strongly negative irreducible one. From this it follows easily that an isolated unramified representation must be 
strongly negative (recall that for $SL(n)$, $n\ne2$, the converse also holds).

Now we  describe the parameters of unramified  irreducible 
strongly negative representations of $Sp(2n,F)$. They  are  pairs of partitions $(p_1,p_2)$, where each $p_i$ is a partition of $k_i$ into different odd positive integers such that
\begin{enumerate}

\item[$\bullet$] $k_1+k_2=2n+1$;

\item[$\bullet$] $p_2$ has even number of terms (i.e., $p_1$ has odd number of terms)\footnote{These parameters are in bijection with discrete unramified  admissible homomorphism of the Weil-Deligne group - see \cite{Mu-unit}.}.

\end{enumerate} 
In  \cite{T-auto} it is described how one  constructs in a simple way representations attached to these parameters.

Further, by \cite{MuT} isolated representation in $ \widehat{Sp(2n,F)}^{\mathbf 1}$  are parameterized by pairs of partitions $(p_1,p_2)$, which satisfy the above two conditions, and also the following two:
\begin{enumerate}
\item[$\bullet$] neither $p_1$ nor $p_2$ contains  consecutive odd numbers;

\item[$\bullet$]  3 is neither in $p_1$ nor in $p_2$\footnote{We have similar description for the special odd-orthogonal groups, where we deal with partitions into different even positive integers}.

\end{enumerate}

 In \cite{T-auto}, we have obtained   combinatorial formulas for the number of above two classes of representations (i.e.,  irreducible unramified strongly negative and isolated ones). Instead of writing these formulas, we  write here the numbers that we get for  $Sp(340)$:

\begin{example}
\label{ex-sp}
 We have
\begin{enumerate}

\item 
$
\text{Number of isolated representations in }
\qquad \quad \ \ \ \widehat{Sp(340,F)}^{\mathbf 1}= \
 11\ 322\ 187\ 942.
$

\item 
$
\text{Number of strongly negative representations in }
\widehat{Sp(340,F)}^{\mathbf 1}=
568\ 385\ 730\ 874.
$

\item G. Mui\'c has proved in \cite{Mu-unit} that each unramified irreducible  strongly negative representation is automorphic (in this way he also proved their unitarity). Therefore, (1) provides us with a huge number of isolated representations in the Ramanujan dual. From this it follows that in the Ramanujan dual 
$\widehat{Sp(340)}^{\mathbf 1}_{v,\text{aut}}$ we have at least
$11\ 322\ 187\ 942
$   isolated representations (and this is the number of  isolated representations in $\widehat{Sp(340)}^{\mathbf 1}_{v,\text{aut}}$ known to us). 
\end{enumerate}
\end{example}

This example shows several new phenomena when compared to the $SL(n)$-case, and raises some questions. A new fact is that we have a huge number of isolated representations in the Ramanujan dual, and also in the unramified unitary dual; also we have a huge number of  unramified irreducible  strongly negative representations. Further new fact is that the number of unramified irreducible strongly negative representation is much bigger then the number of isolated unramified representations (and also then the number of known isolated representations in the Ramanujan dual).

It is interesting that Conjecture ``Arthur + $\epsilon$" of L. Clozel from \cite{Cl-Park} would imply that the set of isolated representations in  $\widehat{Sp(340)}^{\mathbf 1}_{v,\text{aut}}$ is equal to the set of strongly negative representations in $\widehat{Sp(340,F)}^{\mathbf 1}$ (see \cite{T-auto}).

The first question which arises related to the above example is of an arithmetic nature.
D. Kazhdan's proof in \cite{Ka} that the trivial representation is isolated in the unitary dual  of a simple group of rank different from 1  had important arithmetic consequences. Does the above stunning difference regarding the number of isolated points for $SL(n)$ and $Sp(2n)$ groups have some arithmetic explanation, or consequence?

The second question which arises is related to  harmonic analysis: why is such a  small portion of strongly negative representations  isolated (this was not the case for $SL(n)$)? The reason is that for each of $568\ 385\ 730\ 874$ strongly negative unramified representations, excluding $11\ 322\ 187\ 942$ of them (i.e., the isolated ones in the unramified dual), 
 there is a complementary series at whose end  this representation lies. As it is well known,  complementary series representations start with an irreducible representation parabolically induced from a unitary one (and which satisfies additional symmetry conditions with respect to the Weyl group). Therefore, in the example of $Sp(340,F)$, for the $557\ 063\ 542\ 932$ parabolically induced representations which are involved (where corresponding complementary series start),  we need to know their irreducibility.

The following  question is related: why are there still $11\ 322\ 187\ 942$   isolated representations? The answer is roughly: no  complementary series ends with them.

All this tells us that we need to have very explicit knowledge of complementary series (not only on some algorithmic level). This means that we need also to  have a very explicit understanding of the question of irreducibility/reducibility of parabolically induced representations from unramified unitary ones.
Such an understanding is obtained by G. Mui\'c in \cite{Mu-no-un}. Instead of explaining it here, we shall go to a different (and dual) setting in the following section, where we  explain how one can get  a similar  understanding.

Note that for general linear groups we have a perfect understanding of the question of irreducibility/reducibility of parabolically induced representations from  unitary ones. It is given  by J. Bernstein in \cite{Be-P-inv}. The answer is very simple: we always have  irreducibility.

For  other classical groups, the above answer is far from being true. Roughly, in the cases that we  consider and when  reducibility can happen, we  have reducibility in about  ``half" the cases. It  can be different from "half" and the portion for which the reducibility is different from "half" will be crucial information. We  try to explain this in the following section.

\section{Controlling tempered reducibility $\leadsto$ packets of square integrable representations}
\label{s-pack}
\setcounter{equation}{0}
\setcounter{footnote}{0}

We shall start this section  with some questions of harmonic analysis, and see how they bring us  to some deep questions of the theory of automorphic forms.
We  deal with classical groups like $Sp(2n,F)$ or (split) $SO(2n+1,F)$ ($F$ is a non-archimedean field). The main object of this section will be square integrable representations. Recall that they can be characterized as isolated representations in the tempered dual.

Suppose that we are interested in the unitary duals of these groups. Standard strategy  to classify them is to use the non-unitary duals. The non-unitary dual of a reductive group $G$ over $F$ is the set of all equivalence classes of irreducible smooth representations of $G$\footnote{Some authors use the term smooth dual or admissible dual for the non-unitary dual. We prefer (more traditional) term non-unitary dual. This term has been used more often earlier (among others, by J.~M.~G.~Fell in \cite{Fe2}). Here the term ``non-unitary" is used to stress that unitarity of representations is not required in the definition of this dual (the non-unitary dual contains the unitary dual).}. The second part of this strategy is to identify in the non-unitary dual unitarizable classes (i.e. irreducible smooth representations  which admit  $G$-invariant inner products).

  Langlands' classification of the non-unitary duals reduces the non-unitary duals to the tempered duals of Levi subgroups. Since a Levi subgroup in a classical group is a direct product of general linear groups and of a classical group, and since tempered duals of general linear groups are classified (we have commented  this earlier in this paper), the non-unitary duals are in this way  reduced to the tempered duals of classical groups.

Further, the theory of $R$-groups (\cite{Go}) reduces the problem of tempered duals of classical groups to the question of reducibility of the representations
\begin{equation}
\label{temp1}
\text{Ind}_P^G(\d\o\pi),
\end{equation}
 where $\d$ and $\pi$ are irreducible square integrable representations of a general linear group and a classical group respectively, and to the problem of  classification of irreducible square integrable representations of classical groups (in \eqref{temp1}, $G$ is an appropriate classical group, and $P$ is an appropriate parabolic subgroup of $G$)\footnote{Actually, for a more explicit understanding of tempered duals we would need a little bit more  than what the $R$-groups give (see \cite{Jn-temp}, \cite{T-temp} where this is obtained). We do not go  into  detail here.}. We  concentrate now on  the first question,  the question of reducibility of \eqref{temp1}.
We can have reducibility only if $\d$ is selfdual, i.e., if it is  equivalent to its own contragredient. Therefore, we  always assume    that $\d$ is selfdual in what follows.

Each  selfdual irreducible square integrable representation of a general linear group is of the form $\d(\rho,m)$ for some cuspidal selfdual representation $\rho$ and some positive integer $m$ (see  section \ref{s-gl} for notation). In what follows, we assume    that $\rho$ is selfdual. Therefore, we need to understand reducibility of
\begin{equation}
\label{temp2}
\text{Ind}_P^G(\d(\rho,m)\o\pi).
\end{equation}

To get an idea of how the reducibility of \eqref{temp2} behaves, we shall look at some simple examples. For this, one of our old results from \cite{T-irr}  will be useful. To state this result,  we need  to introduce some notation.

Let $\rho$ be irreducible cuspidal representation  of $GL(n,F)$ and $m$ a non-negative integer. Set 
$$
[\rho,|\det|_F^{m}\rho]=\{\rho,|\det|_F\rho,\dots,|\det|_F^{m}\rho\}.
$$
Then $\D=[\rho,|\det|_F^{m}\rho]$ is called a segment in cuspidal representations of general linear groups.
The representation
\begin{equation}
\label{seg}
\text{Ind}_P^{GL((m+1)n,F)}(|\det|_F^{m}\rho\o |\det|_F^{m-1}\rho\o\dots\o \rho)
\end{equation}
contains a unique irreducible subrepresentation, which is denoted by 
$$
\d(\D)
$$
 ($P$ is the appropriate standard  parabolic subgroup).
This is an essentially square integrable representation\footnote{I.e., it becomes square integrable after twisting by an appropriate character}, and  J. Bernstein has shown that one gets all such representations in this way. Note that our former notation (for $\rho$ unitary) becomes
\begin{equation}
\label{seg-cent}
\d(\rho,m)=\d([|\text{det}|_F^{-(m-1)/2}\rho,|\text{det}|_F^{(m-1)/2}\rho]).
\end{equation}
To simplify discussion, we  assume  char$(F)=0$ below. We  now recall   Theorem 13.2 from \cite{T-irr}:

\begin{theorem} \label{th-re-cus} Let $\d(\D)$ be an essentially square integrable representation of a general linear group and $\s$ an irreducible cuspidal representation of a classical group. Then
\begin{equation}
\label{eq-ind-se}
\text{Ind}_P^G(\d(\D)\o\s)
\end{equation}
reduces if and only if
\begin{equation}
\label{eq-ind-cu}
\text{Ind}_{P'}^{G'}(\rho\o\s)
\end{equation}
reduces for some $\rho\in \D$ ($G$ and $G'$ are  appropriate classical groups, and $P$ and $P'$ are appropriate parabolic subgroups in $G$ and $G'$, respectively).
\end{theorem}

The above theorem was proved under a condition  in \cite{T-irr} (which was fulfilled if $\s$ is generic thanks to the fundamental results of \cite{Sh1}; see also \cite{Sh2}). This condition now follows from \cite{A-book}. Note that this book is still conditional, but this is expected to be removed very soon (when the facts on which  \cite{A-book} relies become available). Therefore, we shall use  \cite{A-book} in what follows  without mentioning that there is still a piece to be completed. 

In what follows, $\rho$ will always be an irreducible selfdual cuspidal representation of a general linear group and $\pi$ will be always an irreducible square integrable representation of a classical group (a series of classical groups will be fixed).
Below we shall fix $\rho$ and $\pi$, and consider the reducibility of \eqref{temp2} depending on $m$. We   first look  at some very simple examples. 

The first example  gives a very regular picture.
We shall consider (for a moment) special odd-orthogonal groups. The simplest setting is if we take for the
square integrable representation $\pi$ the trivial representation $1_{SO(1,F)}$ of the  trivial group, and for $\rho$ the trivial representation  $1_{F^\t}$ of $GL(1,F)$.

\begin{example}
\label{ex-triv}
 Recall that $\text{Ind}_{P_\emptyset}^{SO(3)}(|\ |_F^\a)=\text{Ind}_{P_\emptyset}^{SO(3)}(|\ |_F^\a\o 1_{SO(1,F)})$, $\a\in \mathbb R$, reduces  $\iff \a=\pm1/2$. Now Theorem \ref{th-re-cus} implies
$$
\text{Ind}_{P}^{SO(2m+1,F)}(\d(1_{F^\t},m)\o1_{SO(1,F)}) \text{ \ is \  }
\begin{cases}
\text{\ reducible for all even $m$},
\\
\text{\ irreducible for all odd $m$}, 
\end{cases}
$$
since $
|\ |_F^{\pm1/2}\in [\ |\ |_F^{-(m-1)/2},|\ |_F^{(m-1)/2}] \iff m\text { is even}
$
($P$ is the Siegel parabolic subgroup above).
\end{example}

We  now go to the symplectic counterpart, where we  get a slightly less regular picture.

\begin{example}
\label{ex-triv+}
Theorem \ref{th-re-cus} gives
$$
\text{Ind}_P^{Sp(2m,F)}(\d(1_{F^\t},m)\o1_{Sp(0,F)}) \text{ \quad is \quad  }
\begin{cases}
\text{\ reducible for all odd $m$ {\bf except} $m=1$},
\\
\text{\ irreducible for all even $m$} ,
\end{cases}
$$
since  $|\ |_F^{\pm1}\in [\ |\ |_F^{-(m-1)/2},|\ |_F^{(m-1)/2}] \iff m\text{ is odd and } m\ne1$ (reducibility of principal series  of $SL(2,F)$ that we consider is at $\pm1$).
\end{example}
Therefore, $\boxed{1_{F^\t}=\d(1_{F^\t},1)}$ is an exception among the representations $\d(1_{F^\t},m)$.

\begin{remark}
\label{rm-Shah}
F. Shahidi had proved in \cite{Sh2}   that if $\rho\not\cong 1_{F^\t}$, and if  
$
\text{Ind}_{P}^G(|\text{det}|_F^\a\rho \o 1)
$
 reduces, then either $a=0$ or $\a=\pm1/2$.
Now Theorem \ref{th-re-cus}, together with a general reduction to the computation in ``cuspidal lines" from \cite{Jn-sup}, implies
$$
\text{Ind}_{P}^{G}(\d(\rho,m)\o1) \text{ \quad is \quad  }
\begin{cases}
\text{\ reducible for all  $m$ from one parity,}
\\
\text{\ irreducible for all  $m$ from the other parity} .
\end{cases}
$$

\end{remark}

The parity of reducibility is even (resp., odd) if  Shahidi's reducibility is $1/2$ (resp., $0$).
We do not further  discuss here  parity of reducibility/irreducibility (see \cite{Sh2}), but it is related to the local Langlands correspondence for general linear groups (or analytic properties of corresponding $L$-functions). Clearly, the reducibility/irreducibility depends on the series of classical groups that we consider (F. Shahidi has obtained a duality among these reducibilities for the series of groups that we consider; see \cite{Sh2}).

 Observe that the picture that we have gotten from  two last examples is pretty nice. We have  only one exception which does not fit the general even - odd pattern.
Actually, this exception is the first  of a family of examples, when one takes for $\pi$ the Steinberg representation $St_{Sp(2k,F)}$ of $Sp(2k,F)$. 
For $\pi=St_{Sp(2k,F)}$, we can not get  the reducibility of \eqref{temp2} from Theorem \ref{th-re-cus}, but one can get it for example, using some simple principles from \cite{T-irr} based on Jacquet modules (which were used in the proof of Theorem \ref{th-re-cus})\footnote{For dealing with Jacquet modules, the structure obtained in \cite{T-Str} simplifies considerations.}.  We get

\begin{example}
\label{ex-3}  
$$
\text{Ind}_P^{Sp(2(m+k),F)}(\d(1_{F^\t},m)\o St_{Sp(2k,F)}) \text{  is    }
\begin{cases}
\text{\ reducible for all odd $m$ {\bf except} $m=2k+1$},
\\
\text{\ irreducible for all even $m$} .
\end{cases}
$$
\end{example}
So, among representations $\d(1_{F^\t},m)$, $\boxed{\d(1_{F^\t},2k+1)}$ is an exception.
For $\rho\not\cong 1_{F^\t}$ we  have the same situation as in Remark \ref{rm-Shah}, i.e., there are no exceptions.

In general, for general square integrable $\pi$, there can be more than one exception. For example, let $\psi$ be a character of order two of $F^\t$. Then 
\begin{equation}
\label{eq-ord2}
\text{Ind}_{P_\emptyset}^{Sp(4,F)}(|\ |_F\psi\o\psi)
\end{equation}
 contains precisely two irreducible square integrable subquotients. For each of them, the set of exceptional representations (as above) is
\begin{equation}
\label{eq-ord2jb}
\{\d(1_{F^\t},1),\d(\psi,1),\d(\psi,3)\}
\end{equation}
(so they show up for $\rho=1_{F^\t}$ and for $\rho=\psi$). 

For general square integrable $\pi$, the set of such exceptional representations in the above sense  (with respect to $\pi$) is denoted by\footnote{Here we use a slightly different notation than in  \cite{Moe-T} and the other papers, where elements of $\text{Jord}(\pi)$ were pairs $(\rho,m)$ instead of square integrable representations $\d(\rho,m)$  (recall  that  $(\rho,m)\leftrightarrow \d(\rho,m)$ is a bijection by \cite{Z}). One can find  the original definition of C. M\oe glin in \cite{Moe-Ex}.}
$$
\text{Jord}(\pi)\footnote{Roughly, the set of  "singularities" which happen in tempered induction related to $\pi$ is just $\text{Jord}(\pi)$.}.
$$
In other words: 

\begin{definition}
\label{def-jord}
 For an irreducible square integrable representation $\pi$ of a classical group,
 $\text{\rm Jord}(
\pi)$ is the set of all  selfdual  irreducible square integrable representations $\d(\rho,m)$ of general linear groups such that 
$$
\text{Ind}_{P'}^{G'}(\d(\rho, m) \o \pi)
$$
 is irreducible, and  
 $$
\text{Ind}_{P''}^{G''}(\d(\rho, m+2k) \o \pi)
$$
 reduces for some positive integer $k$.
 \end{definition}
 
Above,  $G'$ and $G''$ are appropriate classical groups, while $P' $ and $P''$ are appropriate parabolic subgroups of $G'$ and $G''$ respectively.

After the above definition of $\text{Jord}(\pi)$ we can describe the reducibility of \eqref{temp2} in the following way:

\begin{enumerate}
\item If $\d(\rho,m')\in \text{Jord}(\pi)$ for some $\rho$ and $m'$, then we have reducibility of \eqref{temp2} for $m$'s of that parity, excluding $\d(\rho,m)$'s from $ \text{Jord}(\pi)$, and we have irreducibility in the other parity. 

\item For the remaining $\rho$'s (not showing up in $\text{Jord}(\pi)$), we have reducibility in one parity and irreducibility in the other parity. The parity does not depend on $\pi$. The reducibility parity is even if and only if $\text{Ind}_{P}^G(|\det|_F^{1/2}\rho\o1)$ reduces ($P$ is the Siegel parabolic subgroup).

\end{enumerate}

Therefore, it is crucial  to know $\text{Jord}(\pi)$.

\begin{remark}
\label{rm-cus-p}
Theorem \ref{th-re-cus} directly implies the following three consequences for an  irreducible {\bf cuspidal} representation $\pi$:

\begin{enumerate}
\item
\begin{equation}
\label{eq-no-gap}
\d(\rho,m)\in \text{\rm Jord}(\pi), m\geq 3 \implies \d(\rho,m-2)\in \text{\rm Jord}(\pi).
\end{equation}

\item
We can read from $\text{\rm Jord}(\pi$)   the cuspidal reducibility points which are  different from $0,\pm1/2$. Namely if $\{m;\d(\rho,m)\in \text{\rm Jord}(\pi)\}\ne \emptyset$, let 
$$
a_{\rho,\pi}=\max \{m;\d(\rho,m)\in \text{\rm Jord}(\pi)\}.
$$
Then
\begin{equation}
\label{eq-jb-red}
\text{Ind}_P^G(|\det|_F^{(a_{\rho,\pi}+1)/2}\rho\o\pi) \text{ reduces}.
\end{equation}

\item
We can read information the other way, i.e., if $\text{Ind}_P^G(|\det|_F^{x}\rho\o\pi)$ reduces for some $x\in (1/2)\Z$, $x\geq 1$, then 
\begin{equation}
\label{eq-red-j}
\d(\rho,2x-1), \d(\rho,2x-3), \dots, \d(\rho,\epsilon) \in \text{Jord}(\pi),
\end{equation}
where $\epsilon=1$(resp., 2) if $x$ is an integer (resp.,  not an integer).
 Further,  these are the only members of the form $\d(\rho,m)$ in $\text{\rm Jord}(\pi)$. 

\end{enumerate}
\end{remark}

Therefore, for cuspidal $\pi$, knowing the cuspidal reducibilities $\geq 1$  (which are in $(1/2)\Z$)\footnote{This is always the case by the recent results of J. Arthur, C. M\oe glin and J.-L. Waldspurger.} is equivalent to knowing the Jordan blocks of $\pi$. Other cuspidal reducibilities (i.e., those  at 0 and 1/2) do not depend on $\pi$, but only on the series of groups (and clearly on $\rho$).

In the drawing below, $x$ is the cuspidal reducibility exponent, as in (3) above (in our example below, $x$ is integral), bold segments represent $\d(\rho,m)$'s in the Jordan blocks (since for them $[|\det|_F^{-(m-1)/2}\rho, |\det|_F^{(m-1)/2}\rho]$   does not contain $|\det|_F^{x}\rho$; see \eqref{seg-cent}), and dashed segments represent $\d(\rho,m)$'s which give reducibility (since they contain $|\det|_F^{x}\rho$), and therefore   are not  in the Jordan blocks of $\pi$:

\hskip20mm
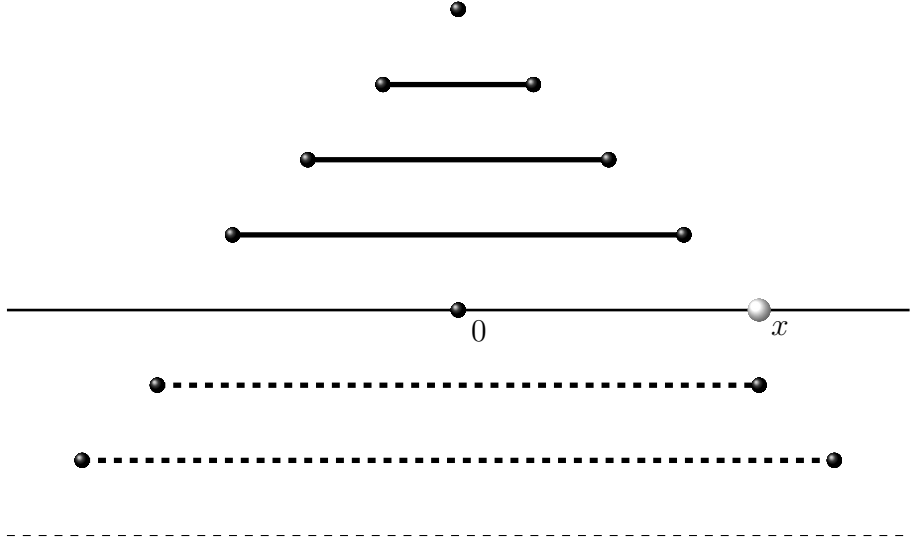
\begin{tikzpicture} %[scale=0.5]
\draw[line width=2pt] (5,6) -- (7,6);
\draw[line width=2pt] (4,5) -- (8,5);
\draw[line width=2pt] (3,4) -- (9,4);
\draw[line width=1pt] (0,3) -- (12,3);
\draw[style=dashed,line width=2pt] (2,2) -- (10,2);
\draw[style=dashed,line width=2pt] (1,1) -- (11,1);
\draw[style=dashed] (0,0) -- (12,0);
\draw (10,3) node[below] {\ \ \ \ $x$};
\draw (6,3) node[below] {\ \ \ \ $0$};
\shade[shading=ball, ball color=black] (6,7) circle (.1);
\shade[shading=ball, ball color=black] (7,6) circle (.1);
\shade[shading=ball, ball color=black] (5,6) circle (.1);
\shade[shading=ball, ball color=black] (8,5) circle (.1);
\shade[shading=ball, ball color=black] (4,5) circle (.1);
\shade[shading=ball, ball color=black] (9,4) circle (.1);
\shade[shading=ball, ball color=black] (3,4) circle (.1);
\shade[shading=ball, ball color=black] (6,3) circle (.1);
\shade[shading=ball, ball color=white] (10,3) circle (.15);
\shade[shading=ball, ball color=black] (10,2) circle (.1);
\shade[shading=ball, ball color=black] (2,2) circle (.1);
\shade[shading=ball, ball color=black] (11,1) circle (.1);
\shade[shading=ball, ball color=black] (1,1) circle (.1);
 \end{tikzpicture}
 
 The above drawing is a graphical interpretation of Remark \ref{rm-cus-p}.  The cuspidal reducibility point (exponent) $x$ determines which segments are ``bold" (i.e., belong to $\text{Jord}(\pi)$), and also conversely: if ``bold" segments are given (i.e., $\text{Jord}(\pi)$ is given), then we can read from this where  the cuspidal reducibility point $x$ is placed.

\begin{definition}
Jordan blocks (of not necessarily cuspidal representations)  which satisfy \eqref{eq-no-gap} will be called Jordan blocks without gaps.
\end{definition}

\begin{remark}
\begin{enumerate}
\item
Observe that (1) of Remark \ref{rm-cus-p} above tells us that Jordan blocks of cuspidal representations do not have gaps.

\item Theorem \ref{th-re-cus} and \cite{Sh1} imply that Jordan blocks of cuspidal generic representations consist only of cuspidal representations (since the cuspidal reducibilities in this case can be only in  $\{0,\pm 1/2,\pm1\}$ by \cite{Sh1}; see the above drawing).

\end{enumerate}
\end{remark}

Up to now in this section, our point of view was completely that  of  harmonic analysis. 
From the other side, square integrable representations  
are important for a number of problems in  automorphic forms. One of the problems is to determine  the local Langlands correspondence for the irreducible square integrable representations that we considered. 

Briefly,
 the  local Langlands correspondence (for the cases that interest us)
  attaches  to an irreducible square integrable representation $\pi$ of a
split connected semi-simple group $G$ over $F$,
a conjugacy class   of continuous
homomorphisms  
$$
\varphi: W_F\times SL(2,\mathbb C)\rightarrow \ ^LG^0,
$$
 such that $\varphi$ maps elements of the first factor to  semi simple elements, it is
algebraic on the second factor, 
and its image is not contained in any proper Levi subgroup of $^LG^0$ (such homomorphisms are called discrete admissible homomorphisms). Above,   $W_F$ denotes the Weil
group of $F$ and $^LG^0$  the complex  dual Langlands $L$-group. The mapping which sends $\pi$ to $\varphi$ will be denoted by
$\Phi_G$, or simply by $\Phi$. Dual Langlands  $L$-groups that interest us are 
$$
\aligned
^L Sp(2n,F)^0
&=SO(2n+1,\mathbb C) ,
\\
^L SO(2n+1,F)^0
&=Sp(2n,\mathbb C).
\endaligned
$$
 Further,  elements with the same admissible homomorphism $\varphi$ should be parameterized  by the equivalence classes of
irreducible representations of the  component  group
\[
\text{Cent}_{^LG^0}(\text{Im}(\varphi))/\text{Cent}_{^LG^0}(\text{Im}(\varphi))^0\ Z(^LG^0)
\]
(Cent$_{^LG^0} X$  denotes the centralizer of  $X\subseteq \,  ^L G^0$,  (Cent$_{^LG^0} X)^0$
 the connected component of  the identity and $Z(^LG^0)$ the center of $^LG^0$). For classical groups, component groups are commutative. Therefore, their irreducible representations are characters.
We shall also use local Langlands correspondences for general linear groups, where 
$$
^L GL(n,F)^0=GL(n,\mathbb
C)
$$
 and where  component groups are trivial.

Recall that to square integrable representation $\pi$ of a classical group,  we have explained how to attach the set $\text{Jord}(\pi)$ of square integrable representations of general linear groups. For general linear groups, the local Langlands correspondences are   known for a while (they were established by G. Laumon, M. Rapoport and U. Stuhler in the positive characteristic, and by M. Harris and R. Taylor  and by G. Henniart in the characteristic 0). Therefore,  we can try to apply it to $\text{Jord}(\pi)$, and see what we get.

Let us go to  some  examples. Consider $\pi=St_{Sp(2k,F)}$, for which we have observed that $\text{Jord}(\pi)=\{\d(1_{F^\t},2k+1)\}$. Applying the local Langlands correspondence $\Phi$ for general linear groups to the only term of $\text{Jord}(\pi)$, we get an admissible homomorphism
$$
1_{W_F}\o E_{2k+1}: W_F\t SL(2,\mathbb C)\rightarrow GL(2k+1,\mathbb C),
$$
where $E_{2k+1}$ denotes the irreducible $2k+1$-dimensional algebraic representation of $SL(2,\mathbb C)$.
Since irreducible algebraic odd-dimensional representations of $SL(2,\mathbb C)$ are orthogonal, the above admissible homomorphism actually  goes into
$$
1_{W_F}\o E_{2k+1}: W_F\t SL(2,\mathbb C)\rightarrow SO(2k+1,\mathbb C).
$$
Note that this is exactly where the admissible homomorphism given by the local Langlands correspondence for symplectic groups should go. 

Consider now irreducible square integrable subquotients of $\eqref{eq-ord2}$. The admissible homomorphism corresponding to them should go to $SO(5,\mathbb C)$. The Jordan blocks here are given by \eqref{eq-ord2jb}. Observe that neither of the homomorphisms  $\Phi(\d(1_{F^\t},1)),\Phi(\d(\psi,1)),\Phi(\d(\psi,3))$ goes into $SO(5,\mathbb C)$. On the other side, their direct sum 
$$
\Phi(\d(1_{F^\t},1))\oplus \Phi(\d(\psi,1))\oplus\Phi(\d(\psi,3))
$$
goes into $SO(5,\mathbb C)$.
Actually, in the third section of \cite{Mu-IMRN}, G. Mui\'c has conjectured that the admissible homomorphism corresponding to an irreducible square integrable generic representation $\pi$ should be the direct sum
$
\oplus\ \Phi(\s),
$
where the sum runs over $\s\in \text{Jord}(\pi)$.

In Theorem 1.5.1 of \cite{A-book}, J. Arthur has obtained  a classification of irreducible square integrable representations of classical groups  (actually, he has obtained classification of tempered representations, but it is easy to single out the square integrable ones).
He has attached to an irreducible square integrable representation $\pi$ of a classical group a pair of an admissible homomorphism and a character of the component group of that homomorphism. 

A fundamental result of C. M\oe glin (Theorem 1.3.1 of \cite{Moe-mult-}) is the following: 

\begin{theorem}
\label{th-Moe}
 The admissible homomorphism that J. Arthur has attached to square integrable representation  $\pi$ is 
\begin{equation}
\label{eq-sum-L}
\underset{\s\in \text{Jord}(\pi)}\oplus\Phi(\s).
\end{equation} 
\end{theorem}
The above homomorphism is discrete, i.e., its image is not contained in any proper Levi subgroup of the (complex) Langlands dual group. Using the above formula, it is equivalent to work with discrete admissible homomorphisms  and Jordan blocks. In what follows, we work with Jordan blocks.

In other words, one of the parameters by which J. Arthur classifies square integrable representations gives  crucial information for  tempered induction in a simple way.
Roughly, we can define a packet as the representations which have the same tempered reducibility properties. Despite the same tempered reducibility,  the representations can be  pretty different  (see the examples in   section \ref{examples}). 

We  end this section (in which we were considering square integrable representations, which can be characterized as isolated points in the tempered duals) with a short note about the Speh representations (which are isolated representations in the automorphic duals). 
Note that we  have not mentioned Speh representations in this section up to now. Nevertheless, they play a significant  role in  Arthur's book \cite{A-book}, which is crucial  to us in this   and the following sections. 
The formula expressing   Speh representation in terms of standard modules plays important role there (see page 427 of \cite{A-book})\footnote{This formula is also important for the global (and the local) Jacquet-Langlands correspondences (which are  instances of functoriality; see \cite{Bad-Inv}, \cite{BR-arch} and \cite{T-div-a2})}. The formula is surprisingly simple, and we briefly present it  here. 

\begin{remark}
\label{Speh}
 The Speh representation $u(\d,m)$ is the unique quotient of \eqref{eq-ind}. Write each tensor factor of the inducing representation in \eqref{eq-ind} as
$$
|\det|_F^{(m-1)/2+k-1}\s=\d\big(\big[|\det|_F^{b_k}\rho,|\det|_F^{e_k}\rho\big]\big), \quad k=1,\dots,m,
$$
where $\rho$ is an  irreducible unitary cuspidal representation. Then in the Grothendieck group of the category of smooth representations we have
\begin{equation}
\label{ch-Speh}
u(\s,m)=\text{det}\left(\left[ \d([|\det|_F^{b_i}\rho,|\det|_F^{e_j}\rho]) \right]_{1\leq i,j\leq m}\right),
\end{equation}
with additional convention that if $b_i=e_j+1$ (resp., $b_i>e_j+1$), we drop the corresponding term $\d([|\det|_F^{b_i}\rho,|\det|_F^{e_j}\rho])$ (resp., we take it to be 0). The multiplication showing up in the determinant is given by  parabolic induction (see \cite{T-ch}, \cite{CR}, \cite{Bad-Speh} or \cite{LMi} for more details)\footnote{The above formula  directly gives an expression for each irreducible unitary representation of a general linear group in terms of standard modules}.
\end{remark}

A formula equivalent to the above one was obtained in \cite{T-ch}. The above simple interpretation is observed in \cite{CR}, which opened way for substantial simplifications of the proof (in \cite{CR}, and in particular in \cite{Bad-Speh}), and a substantial   generalization in \cite{LMi} to the completely  non-unitary setting (to the ladder representations; see \cite{LMi} for the definition). It is interesting to note that this very recent non-unitary generalization has shown already to be very useful even in the unitary setting. Namely, it gives explicit formulas for the derivatives and  Jacquet modules of irreducible unitary representations  in a simple way .

\section{Cuspidals} 
\label{s-inside}
\setcounter{equation}{0}
\setcounter{footnote}{0}

We have briefly discussed Arthur's classification (from  \cite{A-book}) of irreducible square integrable representations of classical groups in the last section.
For a number of questions, it is important to understand what exactly happens in the packets and how the tempered, square integrable and cuspidal representations are related. In the case of general linear groups, all this is solved by the Bernstein-Zelevinsky theory (see \cite{B-Z} and \cite{Z};  some of the main results are mentioned in the previous part of this paper).
For  other classical groups we have   very briefly discussed the relation between tempered and square integrable  representations in the last section.

We shall concentrate here  on the relation  between irreducible cuspidal and square integrable  representations.  A classification of irreducible square integrable representations modulo cuspidal data  is obtained in \cite{Moe-Ex} and \cite{Moe-T}, i.e., modulo cuspidal representations and cuspidal reducibilities (which is equivalent to the knowledge of Jordan blocks by (2) and (3) of Remark \ref{rm-cus-p}). We do not go  into details of this classification here (which can be found in \cite{Moe-Ex} and \cite{Moe-T}). Let us say only  that C. M\oe glin has attached to an irreducible square integrable representation $\pi$ a triple
\begin{equation}
\label{eq-triple}
(\text{Jord}(\pi),\epsilon_\pi,\pi_{cusp}).
\end{equation}
We have defined $\text{Jord}(\pi)$ in the previous section.
The partial cuspidal support $\pi_{cusp}$ is defined as an (equivalence class of) irreducible cuspidal representation(s) of a classical group such that there exists a representation $\theta$ of a general linear group so that $\pi\h\text{Ind}_P^G(\theta\o\pi_{cusp})$. We do not go here into  the definition of the third invariant, a partially defined function $\epsilon_\pi$ (see \cite{Moe-Ex} or \cite{Moe-T} or the introduction of \cite{T-inv}).

These triples satisfy certain conditions (short overview can be found in the introduction of \cite{T-inv}), and triples satisfying these conditions are called admissible triples. Now admissible triples classify irreducible square integrable representations of the series of classical groups that we consider. 
We do not go here into  the definition of admissible triples. Let us only    note  that admissible triples are purely combinatorial objects modulo cuspidal data. Therefore, if in  Arthur's classification we can single out cuspidal representations, this will not only imply the classification of cuspidal representations, but also give cuspidal reducibilities and further imply  an understanding of square integrable representations in term of cuspidal representations (giving in this way an understanding of the internal structure of packets). We  now explain how C. M\oe glin has described the cuspidal representations in  Arthur's classification.

Consider  an irreducible square integrable representation $\pi$ of a classical group. Then  the admissible homomorphism corresponding to $\pi$ is given by \eqref{eq-sum-L}. For $\s \in \text{Jord}(\pi)$, denote by $z_\s$ the linear mapping on the space of \eqref{eq-sum-L}, acting  on the space of $\s$ as the scalar $-1$, and as the identity on the   spaces of $\s'$ for all other $\s'\in \text{Jord}(\pi)$. We  very often use the identification 
\begin{equation}
\label{eq-ident}
\s\leftrightarrow z_\s.
\end{equation}

Now, we  relate the component group and its characters to $\text{Jord}(\pi)$.
If we consider 
$SO(2m+1,F)$ (resp.,  
$Sp(2m,F)$), then
the centralizer in $Sp(2m,
\mathbb C)$ (resp., in $O(2m+1,\mathbb C)$) of the image of \eqref{eq-sum-L}  is a multiplicative group  consisting of elements 
\begin{equation}
\label{additiv}
\prod_{\s \in \text{Jord}(\pi)} z_\s^{a_\s},
\end{equation}
 where $a_\s\in \{0,1\}$. Further, such $a_\s$'s are uniquely determined by \eqref{additiv}. Therefore, we  identify  \eqref{additiv} with the formal (commutative) product
 \begin{equation}
\label{additiv+}
\prod_{\s \in \text{Jord}(\pi),\, a_\s=1} \s.
\end{equation}
Observe that \eqref{additiv+} determines a subset of $\text{Jord}(\pi)$  in an obvious way. In this way, the above centralizer is in a natural bijection with the set  
$2^{\text{Jord}(\pi)}$ of all subsets of $\text{Jord}(\pi)$ (it is an isomorphism when we consider  the operation of symmetric difference on $2^{\text{Jord}(\pi)}$). This is the reason that we shall denote the centralizer by 
$$
2^{\text{Jord}(\pi)}.
$$
Observe that the characters of $2^{\text{Jord}(\pi)}$ are in a natural  bijection with all the functions  $\text{Jord}(\pi)\rightarrow \{\pm1\}$.

 If we consider the special odd-orthogonal group $SO(2m+1,F)$, then $\Phi(\s)$'s in \eqref{eq-sum-L} are symplectic. 
Further, the component group is the quotient of $2^{\text{Jord}(\pi)}$
by the subgroup $\{1, \prod_{\s \in \text{Jord}(\pi)}\s\}$.
Therefore, the characters of the component group can be identified with the characters  of $2^{\text{Jord}(\pi)}$ which are trivial on 
\begin{equation}
\label{so}
\prod_{\s \in \text{Jord}(\pi)} \s.
\end{equation}

Now consider  $Sp(2m,F)$. Then $\Phi(\s)$'s in \eqref{eq-sum-L} are orthogonal. 
The component group is the subgroup of $2^{\text{Jord}(\pi)}$ of all elements of determinant one,
i.e., it consists of all $Y\subseteq \text{Jord}(\pi)$
satisfying 
\begin{equation}
\label{sp}
\prod_{\s \in Y}\text{det}(z_\s)=1.
\end{equation} 
This implies  that if $\d(\rho,2)\in \text{Jord}(\pi)$ (resp., $\d(\rho,a), \d(\rho,b)\in \text{Jord}(\pi)$), then 
$\d(\rho,2)$ 
(resp., $\d(\rho,a) \d(\rho,b))$
 is in the component group  (we deal here with symplectic groups). 
 
 Therefore for both series of groups, after the above identification of characters of the component group, if $\d(\rho,2)\in \text{Jord}(\pi)$ (resp., $\d(\rho,a), \d(\rho,b)\in \text{Jord}(\pi)$), then we can evaluate the characters on elements 
 \begin{equation}
\label{eq-in-cg}
\text{
$\d(\rho,2)$ 
\quad 
(resp., $\d(\rho,a) \d(\rho,b))$.
}
\end{equation}
Observe that for both series of groups, the component group (which we write multiplicatively) is isomorphic in a natural way to a vector space over $\Z/2\Z$. Therefore,  we can talk about a basis of the component group (or $2^{\text{Jord}(\pi)}$),  having in mind this vector space structure.
 
Let $\d(\rho,a)\in \text{Jord}(\pi)$. We define 
$$
a_-=\max \{b; \d(\rho,b)\in \text{Jord}(\pi), b<a\}
$$
if $\{b; \d(\rho,b)\in \text{Jord}(\pi), b<a\}\ne \emptyset$ (otherwise, $a_-$ is not defined).

\begin{definition}
\label{cus-ch}
A character $\varphi $ of the component group corresponding to $\pi$ will be called cuspidal\footnote{C. M\oe glin uses term alternate. Since the same the term is used in \cite{Moe-Ex} and \cite{Moe-T} (in a slightly  different setting), we have rather chosen  the term cuspidal (used by G. Lusztig).}, if it holds
\begin{enumerate}
\item $\text{\rm Jord}(\pi)$ is without gaps;
\item $\varphi(\d(\rho,2))=-1$ whenever $\d(\rho,2)\in \text{\rm Jord}(\pi)$;
\item $\varphi(\d(\rho,a)\d(\rho,a_-))=-1$ whenever $\d(\rho,a)\in \text{\rm Jord}(\pi)$ and $a_-$ is defined.
\end{enumerate}

\end{definition}

Now, Theorem 1.5.1 of \cite{Moe-mult-} tells

\begin{theorem}[C. M\oe glin]
\label{th-Moe-cusp}
An irreducible square integrable representation $\pi$ is cuspidal if and only if $\text{Jord}(\pi)$ is without gaps, and if the  character  of the component group corresponding to $\pi$ is cuspidal.
\end{theorem}

Recall that the above result is again about isolated representations, since (as we have already mentioned) irreducible cuspidal representations can be characterized as the ones that are isolated in the non-unitary dual (see \cite{T-Geo}).

\section{Inside  packets}
\label{examples}
\setcounter{equation}{0}
\setcounter{footnote}{0}

If $\s_1,\dots,\s_k$ are nonequivalent   
irreducible square integrable representations of general linear groups such that all $\Phi(\s_1),\dots,\Phi(\s_k)$ are symplectic, then $\Phi(\s_1)\oplus\dots\oplus\Phi(\s_k)$ is also symplectic.

Let $\s_1,\dots,\s_k$ be nonequivalent   irreducible square integrable representations of general linear groups. Suppose that all $\Phi(\s_1),\dots,\Phi(\s_k)$ are orthogonal. Then $\Phi(\s_1)\oplus\dots\oplus\Phi(\s_k)$ is orthogonal. We want to know when the image goes into the special orthogonal group.
 Write $\s_i=\d(\rho_i,n_i)$. Then
$\Phi(\s_i)=\Phi(\rho_i)\o E_{n_i}$ (recall that by $E_{n_i}$ we denote the irreducible $n_i$-dimensional  algebraic representation of $SL(2,\mathbb C)$). Further, $\Phi(\rho_i)$ is orthogonal if and only if $n_i$ is odd (and $\Phi(\rho_i)$ is symplectic if and only if $n_i$ is even). If $\Phi(\rho_i)$ is symplectic, then $\Phi(\rho_i)\o E_{n_i}$ goes into the special orthogonal group. For orthogonal $\Phi(\rho_i)$, one has   $\text{det}(\Phi(\rho_i)\o E_{n_i})=\text{det}(\Phi(\rho_i)).$ By the properties of $\Phi$, we know $\text{det}(\Phi(\rho_i))=\Phi(\omega_{\rho_1})$. Denote by 
$$
X
$$
 the set of all $i$ such that $\Phi(\rho_i)$ is orthogonal. Therefore, $\Phi(\s_1)\oplus\dots\oplus\Phi(\s_k)$ goes into the special orthogonal group if and only if $\prod_{i\in X} \Phi(\omega_{\rho_i})\equiv 1$, which is equivalent to $\Phi(\prod_{i\in X} \omega_{\rho_i})\equiv 1$, which is further  equivalent to 
\begin{equation}
\label{prod}
\prod_{i\in X}\omega_{\rho_i}\equiv 1.
\end{equation}

The next question that we  try to explain is how to get  the structure of a general packet. A packet is determined by its Jordan blocks (recall \eqref{eq-sum-L}). Therefore, in the symplectic (resp., orthogonal) case, we choose a finite set of non-equivalent orthogonal
(resp., symplectic) irreducible square integrable representations of general linear groups, which we denote by
$$
\text{Jord}.
$$
 In the symplectic case, we also require  that    condition  \eqref{prod}  be satisfied. It is easy to write all the characters of the component group. We  fix one such character (see the previous section), and denote  it by
 $$
 \varphi.
 $$
  Now to get elements of the packet, it is enough to know how to attach the corresponding square integrable representation to a character. We  do this following \cite{Moe-T}
  \footnote{There is no general reference at the moment that the representation that we  attach is the same as the one that Arthur attaches. C. M\oe glin has a proof in \cite{Moe-Pac} for the unitary groups.}.
  The representation that will be attached   recursively in this way to the pair
  $(\text{Jord},\varphi)$ below,
   will be denoted by
  $$
   \lambda_{\text{Jord},\varphi}.
  $$
  
{\bf Recursive construction}:  
\begin{enumerate}  

\item  If the character $\varphi$ is cuspidal (see Definition \ref{cus-ch}), then we take $\lambda_{\text{Jord},\varphi}$  to be the cuspidal representation attached by Arthur to $(\text{Jord},\varphi)$ (see Theorem \ref{th-Moe}). 

{\it In what follows, we assume that $\varphi$ is not cuspidal.}
  
  \item
Suppose that there exists some $\d(\rho,a)\in \text{Jord}$ for which  $a_-$ is defined, and also that 
 \begin{equation}
 \label{=}
 \varphi(\d(\rho,a)\d(\rho,a_-))=1.
 \end{equation}
 Set $\text{Jord}'=\text{Jord}\backslash \{\d(\rho,a),\d(\rho,a_-)\}.$ The character $\varphi$  defines a character of $\text{Jord}'$ in a natural way (by restriction), which we denote by $\varphi'$. Denote by $\pi'$ the square integrable representation attached recursively to $(\text{Jord}',\varphi')$, i.e., $\pi'=\lambda_{\text{Jord}',\varphi'}$. Now the representation
  $$
  \text{Ind}_P^G(\delta([ |\text{det}|_F^{-(a_- -1)/2} \rho, 
|\text{det}|_F^{(a-1)/2}\rho])\o \pi')
$$
has precisely 2 irreducible subrepresentations. Denote them by $\pi_1$ and $\pi_2$. They are not equivalent, and one of them corresponds to $\varphi$. We need to specify which one. 

\begin{enumerate}
\item
Suppose that there exists  $\d(\rho,b)\in \text{Jord}'$ such that  $b_-=a$ ($b_-$ is considered with respect to $\text{Jord}$). We attach to $(\text{Jord},\varphi)$ the representation $\pi_i$ which embeds\footnote{Instead of the embedding requirement here and below, thanks to \cite{Jn-temp} or \cite{T-temp} we can use  a Jacquet module requirement (in this case, the Jacquet module requirement is that $\delta([ |\text{det}|_F^{(a -1)/2+1} \rho, 
|\text{det}|_F^{(b-1)/2}\rho])\o \tau$ is a subquotient of the appropriate Jacquet module of $\pi_i$).} into a representation of the form 
$$
\text{Ind}_P^G(\delta([ |\text{det}|_F^{(a -1)/2+1} \rho, 
|\text{det}|_F^{(b-1)/2}\rho])\o \tau)
$$
if and only if $\varphi(\d(\rho,b)\d(\rho,a))=1$. 

\item 
Suppose that there exists  $\d(\rho,b)\in \text{Jord}'$ such that  $b=(a_-)_-$ ($(a_-)_-$ is considered with respect to $\text{Jord}$). We attach to $(\text{Jord},\varphi)$ the representation $\pi_i$ which embeds 
into a representation of the form 
$$
\text{Ind}_P^G(\delta([ |\text{det}|_F^{(b -1)/2+1} \rho, 
|\text{det}|_F^{(a_--1)/2}\rho])\o \tau)
$$
if and only if $\varphi(\d(\rho,b)\d(\rho,a_-))=1$.  

\medskip

{\it It remains to consider the case when we have no $\d(\rho,b)$ in $\text{\rm Jord}'$.}

\medskip

\item
If $a$ is even, then we attach to $(\text{Jord},\varphi)$ the representation $\pi_i$ which embeds into a representation of the form 
$$
\text{Ind}_P^G(\delta([ |\text{det}|_F^{1/2} \rho, 
|\text{det}|_F^{(a_--1)/2}\rho])\o \tau)
$$
if and only if $\varphi(\d(\rho,a))=1$.  

\item
Suppose that $a$ is odd. Then $\text{Ind}_{P'}^{G'}(\rho\o\pi_{cusp})$ reduces. In \cite{Moe-Ex} and \cite{Moe-T} an  indexing of the irreducible subrepresentations is fixed: $\text{Ind}_{P'}^{G'}(\rho\o\pi_{cusp})=\tau_1\oplus \tau_{-1}$ (when $\pi_{cusp}$ is generic, we shall always take $\tau_1$ to be generic). We attach to $(\text{Jord},\varphi)$ the representation $\pi_i$ determined by the fact that it embeds into  
$$
\text{Ind}_P^G(\theta\o\delta([ |\text{det}|_F \rho, 
|\text{det}|_F^{(a-1)/2}\rho])\o \tau_1)
$$
for some irreducible representation $\theta$ of a general linear group,
if and only if $\varphi(\d(\rho,a))=1$. 

\end{enumerate}

\medskip

{\it In what follows, we can assume that  \eqref{=} does  not occur for $\varphi$.}

\medskip

\item
Suppose that $\text{Jord}$ has gaps. Then there exists $\d(\rho,a)\in \text{Jord}$, $a\ge3$ and $k\in\Z_{>0}$ such that $\d(\rho,b)\not\in \text{Jord}$ for any $b\in[a-2k,a-2]$. Set $\text{Jord}'=\text{Jord}\backslash \{\d(\rho,a)\}\cup \{\d(\rho,a-2k)\}.$ Define the character $\varphi'$ of the component group of $\text{Jord}'$ in a natural way from $\varphi$  (putting $\d(\rho,a-2k)$ instead of $\d(\rho,a)$). Let $\pi'$ be the representation recursively attached to $(\text{Jord}',\varphi')$. Then the representation 
$$
  \text{Ind}_P^G(\delta([ |\text{det}|_F^{(a -2k+1)/2} \rho, 
|\text{det}|_F^{(a-1)/2}\rho])\o \pi')
$$
contains a unique irreducible subrepresentation. This subrepresentation is square integrable, and we attach this subrepresentation to $(\text{Jord},\varphi)$ (this is compatible\footnote{This means that restricting $\varphi$, we get the partially defined function attached to the representation in \cite{Moe-Ex}} with the construction in \cite{Moe-T} by Theorem 8.2 of \cite{T-temp}, or by Corollary 2.1.3 of \cite{Jn-temp}).

\medskip 

{\it In what follows, we can assume that $\text{\rm Jord}$ has no gaps.}

\medskip

\item
The only possibility which remains is that we have some $\d(\rho,2)\in \text{Jord}$ such that $\varphi(\d(\rho,2))=1$. Set $\text{Jord}'=\text{Jord}\backslash \{\d(\rho,2)\}.$ Define the character $\varphi'$ of the component group of $\text{Jord}'$ in a natural way from $\varphi$ (restricting). Let $\pi'$ be the representation recursively attached to $(\text{Jord}',\varphi')$. Then the representation
$$
\text{Ind}_P^G( |\text{det}|_F^{1/2} \rho\o\pi')
$$
has a unique irreducible subrepresentation. This subrepresentation is square integrable, and we attach this subrepresentation to $(\text{Jord},\varphi)$ (this is compatible with the construction in \cite{Moe-T}  by  Lemma \ref{dodatak} in the appendix of this paper, or by Lemma 3.2.1 of \cite{Jn-temp}).

\end{enumerate}

To illustrate what  
packets look like, we   consider some very simple examples, first in the case of
 special odd-orthogonal groups, and later in the case of
 symplectic groups.   We 
 pay special attention to the
 packets simultaneously containing both cuspidal representations as well as  representations supported by the minimal parabolic subgroup\footnote{i.e., which are subquotients of  representations parabolically induced from the minimal parabolic subgroups.}. Such packets will be called packets with antipodes (cuspidal representations in such packets have  simple parameters, and simple cuspidal reducibilities $\geq 1$).

Below, we  call packets containing an Iwahori-spherical representation simply Iwahori packets (clearly, in such a  packet Iwahori-spherical representations are precisely the ones supported by the minimal parabolic subgroup).

The  simple examples below are presented  to give a flavor of what can happen in the packets: we can have representations supported on a number of different parabolic subgroups (often including a relatively small number of cuspidal representations when Jordan blocks are without gaps),  and we can have also packets containing only cuspidal representations.

\bigskip
\begin{example} 
\label{ex-ort}
{\rm In this example  we consider packets  for special odd-orthogonal groups}. 
\end{example}
\begin{enumerate}
\item Let $\psi_1,\psi_2$ be different characters of $F^\t$ satisfying $\psi^2_i\equiv 1$ for $i=1,2$. 
Consider the packet determined by  Jordan blocks  
$$
\text{Jord}=\{\d(\psi_1,2),\ \d(\psi_2,2)\}.
$$
Obviously, the Jordan blocks are without gaps.
We  shall write a character of the  centralizer $2^{\text{Jord}}$ as 
 $$
 \varphi_{(\epsilon_1,\epsilon_2)},
 $$
  which means that this character sends the first element of the basis to $\epsilon_1$ and the second one to $\epsilon_2$ (we shall use this notation for characters in what follows, whenever a basis is fixed).
  If a character $\varphi_{(\epsilon_1,\epsilon_2)}$ of  $2^{\text{Jord}}$ is a character 
of the component group, it must be trivial on $\d(\psi_1,2) \d(\psi_2,2)$.
  Therefore, we have  two characters of the component group.
One of them, $\varphi_{(-1,-1)}$,  is a cuspidal character.   
The corresponding cuspidal representation we have denoted by 
$$
\lambda_{\{\d(\psi_1,2),\d(\psi_2,2)\},\varphi_{(-1,-1)}}
$$
(here we  apply step (1) of the recursive construction).

It remains to consider the trivial character $\varphi_{(1,1)}$. To construct the corresponding representation, we   apply step (4) of the recursive construction two times.
In this way, we get that the corresponding representation is the unique irreducible subrepresentation of
$$
\text{Ind}_{P_\emptyset}^{SO(5,F)}(|\ |^{1/2}_F\psi_1\o|\ |^{1/2}_F\psi_2),
$$
which is also the unique irreducible square integrable subquotient of the 
above representation.

\item Let $\psi_1,\psi_2,\psi_3$ be different characters of $F^\t$ satisfying $\psi^2_i\equiv 1$ for $i=1,2,3$.  Consider the packet determined by Jordan blocks  
$$
\d(\psi_1,2), \ \d(\psi_2,2), \ \d(\psi_3,2).
$$
 Evidently, Jordan blocks  are without gaps but we do not have   cuspidal characters, since the character $\varphi_{(-1,-1,-1)}$ is not trivial on $\d(\psi_1,2)\d(\psi_2,2)\d(\psi_3,2)$.

 One gets three elements of the packet corresponding to the non-trivial characters in the following way (applying step (4) of Recursive construction once for each of these representations).  Choose $i\in\{1,2,3\}$,
and denote remaining two indices  by $j_{i,1}$ and $j_{i,2}$. Now
\begin{equation}
\label{2}
\text{Ind}^{SO(7,F)}_P(|\ |^{1/2}_F\psi_i\o \lambda_{\{\d(\psi_{j_{i,1}},2),\d(\psi_{j_{i,2}},2)\},\varphi_{(-1,-1)}})
\end{equation}
contains a unique irreducible square integrable subquotient.
One gets the character corresponding to \eqref{2}  from $\varphi_{(-1,-1,-1)}$ by putting  $1$ instead of $-1$ at $i$-th place.

 Applying step (4) of the recursive construction, and the previous example, we get that an irreducible
  square integrable  subquotient of 
  \begin{equation}
\label{1}
\text{Ind}_{P_\emptyset}^{SO(7,F)}(|\ |^{1/2}_F\psi_1\o|\ |^{1/2}_F\psi_2\o|\ |^{1/2}_F\psi_3)
\end{equation}
 corresponds to the trivial character $\varphi_{(1,1,1)}$. 
  Further, \eqref{1}
 contains a unique irreducible square integrable subquotient (it is the unique irreducible subrepresentation).

\item Take irreducible cuspidal  representations  $\rho_1,\dots,\rho_k$  of general linear groups. Suppose that all $\Phi(\rho_1),\dots, \Phi(\rho_k)$ are symplectic and non-equivalent. 
Then $\Phi(\rho_1)\oplus\dots\oplus \Phi(\rho_k)$ is a discrete admissible homomorphism. The corresponding set of Jordan blocks is $\{\rho_1,\dots,\rho_k\}$. Observe that these Jordan blocks do not have gaps, and each character of the component group is cuspidal. Therefore, here one  gets  a packet consisting of $2^{k-1}$ cuspidal representations (here we  apply only step (1) of the recursive construction).

\item If a packet (of $SO(2n+1,F)$) contains a representation supported by the minimal parabolic subgroup, then each  element in the Jordan blocks of such packet has the form $\d(\psi,2k$) for some character $\psi$ of $F^\t$ satisfying $\psi^2\equiv 1$, and some $k\in \Z_{\geq 1}$. The converse also holds (consider the trivial character of the component group).

\item
We  now consider a more general packet then the one in (1). Let $\psi_1,\psi_2$ be different characters of $F^\t$ satisfying $\psi_1^2\equiv 1$, $i=1,2$, and $k$ a positive integer. Consider  the packet of $SO(2k(k+1)+1,F)$    determined by the Jordan blocks
 \begin{equation}
 \label{jord-a}
\{\d(\psi_1,2i), \d(\psi_2,2i); i=1,2,\dots ,k\}.
\end{equation}
This packet  obviously has one cuspidal representation. This cuspidal representation is not generic, which  follows from \cite{Sh2} since the above representation has two cuspidal reducibilities $> 1$ (they correspond to $\psi_1$ and $\psi_2$;  both reducibilities are at  $k+1/2$).

\item Suppose that $SO(2n+1,F)$ has an Iwahori packet with antipodes.  
Then the fact that Jordan blocks of packets containing cuspidal representations have no gaps, implies that   
 $2n=k_1(k_1+1)+k_2(k_2+1)$ for some $k_1,k_2\in \mathbb Z_{\geq 0}$ (i.e., $n$ is a sum of two triangular numbers). Such a packet without gaps contains a cuspidal representation if and only if $\lfloor(k_1+1)/2\rfloor+\lfloor(k_2+1)/2\rfloor\in 2\Z$, where here $\lfloor x\rfloor $ denotes the largest integer not exceeding
  $x$.

\end{enumerate}

Very often we shall consider a slightly different problem then the above one: for a given  irreducible square integrable representation $\pi$, describe the remaining representations of the  packet to which  $\pi$ belongs. To solve this, one possibility  
is to find the Jordan blocks of $\pi$, and apply the recursive construction.  Proposition 2.1 
of \cite{Moe-T} is  useful for this (see also Proposition 3.1 of \cite{T-temp}).
We do not go  into the details here.

\begin{example}
\label{ex-symp}
{\rm In this example we consider packets  for symplectic groups.}
\end{example}
\begin{enumerate}

\item 
We  consider the packet given by Jordan blocks 
$$
\text{Jord}=\{\d(1_{F^\t},1), \d(1_{F^\t},3), \d(1_{F^\t},5)\}.
$$
For a basis $B$ of the component group we can take 
$$
\d(1_{F^\t},1)\d(1_{F^\t},3),\ \d(1_{F^\t},3)\d(1_{F^\t},5).
$$
As before, we write characters of the component group as $\varphi_{(\epsilon_1,\epsilon_2)}$, which means that this character sends the first element of the basis to $\epsilon_1$ and the second one to $\epsilon_2$.

To attach  representations to the characters $\varphi_{(\pm1,1)}$, we apply step (2) of the recursive construction: the representation $\text{Ind}_{P_\emptyset}^{Sp(8,F)}(\d([|\ |_F^{-1},|\  |_F^{2}])\o 1_{Sp(0,F)})$ has two irreducible subrepresentations. They are square integrable. 
Only one of them has a subquotient of the form $|\ |_F\o *$  in its Jacquet module. The character corresponding to this one is $\varphi_{(1,1)}$, while $\varphi_{(-1,1)}$ corresponds to the other subrepresentation (here, to attach characters to representations, we have applied (b) of step (2) in the recursive construction).

To attach  representations to the characters $\varphi_{(1,\pm1)}$, we consider the representation $\text{Ind}_{P}^{Sp(8,F)}(\d([|\  |_F^0,|\  |_F])\o St_{Sp(4,F)})$. This representation has two irreducible subrepresentations (here we have applied step (2) and then step (3) of the recursive construction). They are square integrable.  Only one of them does not have subquotient of the form $|\ |_F^2\o *$ in its 
Jacquet module. 
The character corresponding to this one is $\varphi_{(1,-1)}$ (here, to attach character to the representation, we have applied (a) of step (2) in the recursive construction).

 The fourth representation in the packet corresponds to the character $\varphi_{(-1,-1)}$. This character is cuspidal, so the corresponding representation is cuspidal (here we apply step (1) of the recursive construction). We  comment on this representation later. 
  
 Observe that the principal series 
$$
\text{Ind}_{P_\emptyset}^{Sp(8,F)}(|\ |_F^{-1}\o|\ |_F^{0}\o|\ |_F^{1}\o|\ |_F^{2})
$$
has exactly 3 irreducible non-equivalent square integrable subquotients.

\item The following example is related to \eqref{eq-ord2jb}.
We shall consider the packet determined by Jordan blocks  given by $\d(1_{F^\t},1),\d(\psi,1),\d(\psi,3)$ (i.e., by \eqref{eq-ord2jb}).
For the basis $B$ of the component group, we can take 
$$
\d(1_{F^\t},1)\d(\psi,1), \ \d(\psi,1)\d(\psi,3).
$$
Now the square integrable subquotients of \eqref{eq-ord2} correspond to the characters $\varphi_{(\pm1,1)}$ by step (2) of the recursive construction ($\varphi_{(1,1)}$ corresponds to the generic one by  (d) of step (2) of the recursive construction).
 The cuspidal characters are $\varphi_{(\pm1,-1)}$. Thus, here we have  two cuspidal representations in the packet (we apply step (1) of the recursive construction).

\item  Let $\rho$ be an irreducible selfdual cuspidal representation of $GL(2,F)$ with trivial central character\footnote{For this and the following example, we could take $\rho$ to be any irreducible  cuspidal representation of $GL(k,F)$ such that $\Phi(\rho)$ is symplectic.}.   
We now consider  the packet given by Jordan blocks  
$$
\d(1_{F^\t},1),\ \d(\rho,2).
$$
 For the basis $B$ of the component group, we can take 
 $$
 \d(\rho,2).
 $$
 Then $\text{Ind}^{Sp(4,F)}_P(|\text{det}|_F^{1/2}\rho)$ contains precisely one irreducible square integrable  subquotient. It is the  unique irreducible subrepresentation.
 The above square integrable representation corresponds to the trivial character (by step (3) of the recursive construction).
 
  There is one cuspidal character, denoted by $\varphi_{(-1)}$ (it sends $\d(\rho,2)$ to $-1$).  The corresponding cuspidal representation is denoted by 
$
\lambda_{\{\d(\rho,2)\},\varphi_{(-1)}}
$
 (this is step (1) of the recursive construction). Then
\begin{equation}
\label{eq-ex-1}
\text{Ind}^{Sp(8,F)}(|\text{det}|_F^{3/2}\rho\o \lambda_{\d(\rho,2),\varphi_{(-1)}})
\end{equation}
reduces.

\item Let $\rho$ be an irreducible selfdual cuspidal representation of $GL(2,F)$ with trivial central character (as in (3)).  Consider Jordan blocks  
$$
\d(1_{F^\t},1),\ \d(\rho,2),\ \d(\rho,4).
$$
 For the basis $B$ of the component group, we can take 
 $$
 \d(\rho,2),\ \d(\rho,4).
 $$
 Then $\text{Ind}^{Sp(12,F)}(\d([|\text{det}|_F^{-1/2}\rho, |\text{det}|_F^{3/2}\rho]))$ contains precisely two irreducible square integrable subquotients. Moreover, they are the  unique irreducible subrepresentations. They
  correspond to the characters $\varphi_{(1,1)}$ and $\varphi_{(-1,-1)}$. Only one of these representations has a subquotient of the form $|\text{det}|_F^{1/2}\rho\o*$ in its Jacquet module, and this one corresponds to $\varphi_{(1,1)}$ (here we have applied step (2) of the recursive construction, and then (c) in that step). 
  
   Further $\varphi_{(-1,1)}$ is the only cuspidal character here. It corresponds to a cuspidal representation (here we apply step (1) of the recursive construction). 
   
   The fourth element of the packet, corresponding to the character $\varphi_{(1,-1)}$, is the unique irreducible subrepresentation of 
\begin{equation}
\label{eq-ex-2}
\text{Ind}_P^{Sp(12,F)}(|\text{det}|_F^{1/2}\rho\o |\text{det}|_F^{3/2}\rho\o \lambda_{\{\d(\rho,2)\},\varphi_{(-1)}}).
\end{equation}
This representation is not only square integrable, but moreover strongly positive (in the terminology of \cite{Moe-T}). 
Here we have first applied step (4) of the recursive construction, then step (3) which brought us to the cuspidal character, where we apply step (1).

\item Let $\rho_1,\dots,\rho_l$ be non-equivalent   irreducible cuspidal representations of general linear groups, such that all $\Phi(\rho_1),\dots,\Phi(\rho_l)$ are orthogonal and $\prod_{i=1}^l\omega_{\rho_i}\equiv 1$. Then $\Phi(\rho_1)\oplus\dots\oplus\Phi(\rho_l)$ is an admissible homomorphism. The corresponding set of Jordan blocks is $\{\rho_1,\dots,\rho_l\}$. These Jordan blocks do not have gaps, and each character of the component group is cuspidal. Therefore, here one  gets  a packet consisting of $2^{l-1}$ cuspidal representations.

\item    Denote by $\psi_{un}$ the unramified character of $F^\t$ of order 2. Suppose that $2n+1$ is a sum of two squares. Then one of them must be even and the other odd. Denote them by $n_e$ and $n_o$, respectively. Then 
$$
\{(1_{F^\t},2i-1);i=1,2,\dots,n_o\}
\cup
\{(\psi_{un},2i-1);i=1,2,\dots,n_e\}
$$
are parameters of a packet of $Sp(2n,F)$. This packet contains two cuspidal representations. Further, the representation corresponding to the trivial character of the component group is supported on the minimal parabolic subgroup, and it is Iwahori-spherical. Thus, this is an Iwahori packet with antipodes.

From this, it follows that $Sp(2n,F)$ has an Iwahori packet with antipodes    if and only if $2n+1$ is a sum of two squares.

\item If a packet (of $Sp(2n,F)$) contains a representation supported on the minimal parabolic subgroup, then  all  the Jordan blocks of such a packet have the form $\d(\psi,2k-1$) for some characters $\psi$ of $F^\t$ satisfying $\psi^2\equiv 1$, and some $k\in  \Z_{\geq 1}$. 

The converse does not hold.  
For example, take any three characters $\chi_1,\chi_2,\chi_3$ of $F^\t$ or order two such that $\chi_1\chi_2\chi_3=1_{F^\t}$, and take any three odd positive integers $k_1,k_2,k_3$ satisfying $k_1+k_2+k_3=2n+1$. Then the packet of $Sp(2n,F)$ determined by Jordan blocks 
$\d(\chi_1,k_1),\d(\chi_2,k_2),\d(\chi_3,k_3)$ consists of 4 representations. All these representations are supported on  parabolic subgroups whose Levi factors are isomorphic to $GL(1,F)^{n-1}\t Sp(2,F)$.

Related to the above discussion, consider an irreducible cuspidal representation $\rho$ of $GSp(2,F)=GL(2,F)$ which splits into 4 pieces after restriction to $Sp(2,F)=SL(2,F)$. Then a packet of $Sp(2,F)$ consists of all irreducible pieces of the restriction $\rho|_{Sp(2,F)}$. The Jordan blocks are three non-trivial characters $\psi_1,\psi_2,\psi_3$, characterized by the condition $(\psi_i\circ\text{det})\rho\cong\rho$ (all this follows directly from \cite{W1} and Theorem \ref{th-re-cus}).
One can now consider the packet of $Sp(2n,F)$ determined by Jordan blocks $\d(\psi_1,k_1),\d(\psi_2,k_2),\d(\psi_3,k_3)$ and  easily describe its elements in terms of the elements of the packet of $Sp(2,F)$ considered above.

\end{enumerate}

Some questions  related to the packets  
arise naturally. One  question is to determine, for some  
 irreducible cuspidal representations obtained or constructed by other methods,
 the packets  to which they belong. 

From the other side, we have seen in the above examples that starting from  very simple Jordan blocks, we can have cuspidal representations in the packet. In general, such cuspidal representations will be degenerate. Also, they will be rare in the packet. The question is, can one describe (at least some  of) those representations in a different  way? 
 Clearly, these two questions are related. The Howe correspondences are a great source of  representations for the second question. Related to this, we  give some simple examples.

Let us first go to the packet in (1) of Example \ref{ex-symp} (which has very simple Jordan blocks). Using the Howe correspondence for $Sp(8,F)$ and $O(Y)$, where $Y$ is a totally anisotropic orthogonal space of dimension 4, C. M\oe glin has obtained from the signum character of $O(Y)$ an irreducible cuspidal representation $\s$ of $Sp(8,F)$. She has  also gotten  that $\text{Ind}_P^{Sp(10,F)}(|\ |^3_F\o \s)$ reduces (considering the Howe correspondence for $Sp(10,F)$ and $O(Y)$). Now (3) of Remark \ref{rm-cus-p} and Theorem \ref{th-Moe} imply directly that $\s$ is the cuspidal representation which is in the packet in (1) of Example \ref{ex-symp}.

Consider for a moment the packet of $SO(2k(k+1)+1,F)$ in  Example \ref{ex-ort}, (5), when one takes $\psi_1$ and $\psi_2$ unramified (then one of them is $1_{F^\t}$ and the other one is the unramified signum character; this is an Iwahori packet).
Clearly, it is natural to expect that the cuspidal unipotent $SO(2k(k+1)+1,F)$-representation of G. Lusztig  belongs  to this packet.

Now consider the packet in Example \ref{ex-symp}, (6), when one takes $2n+1$ to be the sum of two consecutive squares $k^2$ and $(k+1)^2$ (then $n$ is twice a triangular number). We get an Iwahori packet of $Sp(2k(k+1),F)$. Again, it is natural to expect that the cuspidal unipotent $Sp(2k(k+1),F)$-representation of G. Lusztig   belongs to this packet.

    Recall that the existence of  Iwahori packets  with antipodes for special odd-orthogonal (resp., symplectic) groups is related to the sums of triangular numbers (resp., sums of squares of integers). We have  discussed above only the case of two equal triangular numbers (resp., sum of two consecutive squares).
It is interesting to find other descriptions of the cuspidal representations in the remaining Iwahori packets. C. M\oe glin's construction, which we have discussed above,  gives a description of such a representation of $Sp(8,F)$  (where $9$ is the sum of $3^2$ and $0^2$).

   These are only some simple questions which arise related to the packets that we have considered.

\begin{remark}
Irreducible representations of compact Lie groups are classified by highest weights (\cite{We}). Other information about these representations may be  obtained from the corresponding highest weight (dimension, character, representation itself, etc.). Similarly, for the irreducible square integrable representations of classical $p$-adic groups, using their  parameters discussed at the beginning of  this section, it would be interesting to get 
  other relevant information about  representations, in particular, for the representations with simple parameters. For this, other descriptions of the representations might be useful.

\end{remark}
 
 We end this section with a simple result:
 
 \begin{proposition} The group $SO(2n+1,F)$  has a packet with antipodes if and only if $n$ is even.   
 
 Suppose that the residual characteristic  of $F$ is odd. Then  $Sp(2n,F)$ has a packet with antipodes if and only if $n$ is even\footnote{Odd residual characteristic is used only for the implication $\implies$.}. 
 
 \end{proposition}

\begin{proof}
First we  consider special odd orthogonal groups.

 Let $X$ be a packet with antipodes of $SO(2n+1,F)$. Since the packet contains a representation supported on the minimal parabolic subgroup, all elements of the Jordan blocks must be of the form $\d(\psi, k)$, where $\psi$ is a character of $F^\t$ satisfying $\psi^2\equiv 1$ and $k$ is even.  Let $\psi_i$, $i=1,\dots,m$, be all such different characters that show up in Jordan blocks, and let $k_i=\max\{l;\d(\psi_i,2l)\in X\}$.  Since we have a cuspidal representation  in the packet, there are no gaps. Therefore
\begin{equation}
\label{so-1}
X=\cup_{i=1}^m \{\d(\psi_i,2i);i=1,\dots,k_i\}
\end{equation}
and thus
\begin{equation}
\label{so-2}
\sum_{i=1}^m k_i(k_i+1)=2n.
\end{equation}
Observe that there can be only one cuspidal character, and it must be trivial on \eqref{so}. This implies
\begin{equation}
\label{so-3}
\sum_{i=1}^m \lfloor (k_i+1)/2\rfloor \in 2\Z,
\end{equation}
where, as before, $\lfloor x\rfloor $ denotes  the largest integer not exceeding
  $x$. Observe that
\begin{equation}
\label{so-4}
 \lfloor (k_i+1)/2\rfloor \in 1+2\Z \iff k_i\equiv 1,2 \  (\text{mod } 4)\iff  k_i(k_i+1) \in 2+4\Z.
\end{equation}

From the other side, if we have (different) $\psi_i$ satisfying $\psi_i^2\equiv 1$, and $k_i\in\Z_{\geq1}$  satisfies  \eqref{so-2} and  \eqref{so-3}, then if we define $X$  by \eqref{so-1}, we get a packet with antipodes.

From the above considerations, we see that in the case of packet with antipodes, \eqref{so-3}, \eqref{so-4} and  \eqref{so-2} imply that $n$ must be even.

Suppose now that $n$ is even. Then $n-1$ is odd. Recall that we have at least 4 different characters $\psi$ satisfying $\psi^2\equiv 1$. Denote them by $\psi_i,i=1,\dots,4$. A classical result of Gauss says that $n-1$ is a sum of three triangular numbers, i.e., $2(n-1)=\sum_{i=1}^3 k_i(k_i+1)$. Now since we have proved that $SO(2(n-1)+1,F)$ does not have packets with antipodes, \eqref{so-3} cannot hold. Thus  $\sum_{i=1}^3 \lfloor (k_i+1)/2\rfloor \in 1+2\Z$. Now  taking $k_4=1$, we get
$$
 \sum_{i=1}^4 k_i(k_i+1)=2n \text{ \ and \ } \sum_{i=1}^4 \lfloor (k_i+1)/2\rfloor \in 2\Z.
$$
The above discussion now implies that $SO(2n,F)$ has a packet with antipodes.

It remains to  consider the symplectic groups. Observe that we have now 4 characters $\psi$ satisfying $\psi^2\equiv 1$ (because the residual characteristic is odd).

 Let $X$ be a packet with antipodes of $Sp(2n,F)$. Since the packet contains a representation supported on the minimal parabolic subgroup, all elements of the Jordan blocks are of the form $\d(\psi, k)$, where $\psi$ is a character of $F^\t$ satisfying $\psi^2\equiv 1$ and $k$ is odd.  Let $\psi_i$,  $i=1,\dots,m$, be all such  characters that show up in the Jordan blocks, and let $k_i=\max\{l;\d(\psi_i,2l-1)\in X\}$. 
   Since we have a cuspidal representation in the packet, there are no gaps. Therefore
\begin{equation}
\label{sp-1}
X=\cup_{i=1}^m \{\d(\psi_i,2l-1);l=1,\dots,k_i\},
\end{equation}
which implies
\begin{equation}
\label{sp-2}
\sum_{i=1}^m k_i^2=2n+1.
\end{equation}
Now, condition $\eqref{sp}$ tells us that
 \begin{equation}
\label{sp-3}
\prod_{i=1}^m \psi_i^{k_i}\equiv1.
\end{equation}
This can happen exactly in two ways:
\begin{enumerate}
\item if $\psi_i\not\equiv 1$, then $k_i$ is even;
\item if $\psi_i\not\equiv 1$, then $k_i$ is odd, and all the three non-trivial characters show up as $\psi_i$'s.
\end{enumerate}

We  now show that (2) cannot happen. Suppose that  (2) holds. Let $\varphi$ be a character  of the component group which corresponds to a representation supported on the minimal parabolic subgroup. Now perform  the step (2) of the recursive construction as long as possible. We  come to a packet of a (possibly smaller) group, where  all  three non-trivial $\psi$ still show up. Now we  perform  step (3) as long as possible. We shall come to a packet of a (possibly smaller) group, where still all the three non-trivial $\psi$ will show up (which implies that it is a packet of some $Sp(2\ell,F)$ with $\ell\geq 1$), but without gaps. The character that we get in this way must be cuspidal. The  representation corresponding to this character is cuspidal. This implies that the representation corresponding to the initial character cannot be supported on the minimal parabolic subgroup. This finishes the proof of the claim. 

Therefore, (1) holds for the packets with antipodes. Because of this and \eqref{sp-2}, $1_{F^\t}$ must always show up in $X$. We  denote $1_{F^\t}$ by $\psi_1$. Then $k_1$ must be odd.

Now \eqref{sp-2} implies that $n$ must be even.

It remains to show that for each even $n$, we can find a packet with antipodes.
The above discussion implies that for this, it is enough to show that for each $l\in\Z_{\geq0}$ we can find $k_1\in1+2\Z_{\geq 0}$ and $k_2,k_3,k_4\in 2\Z_{\geq0}$ such that
$
4l+1=\sum_{i=1}^4 k_i^2.
$
To prove this, it is enough 
 to show that for each $l\in\Z_{\geq0}$ we can find   $m_1,\dots,m_4\in \Z_{\geq0}$ such that
 \begin{equation}
\label{sp-4}
l=m_1(m_1+1)+\sum_{i=2}^4 m_i^2.
\end{equation}
If $l$ is not of the form $4^a(8b+7)$, then a classical result of Gauss tells us that we can do this with $m_1=0$. If $l$ is of the form $4^a(8b+7)$, then we take $m_1=1$ and apply the  above classical result to $l-2$, which is now not of the form $4^a(8b+7)$. This  again gives the representation \eqref{sp-4}. This completes the proof of the proposition.
\end{proof}

\section{Appendix}
\setcounter{equation}{0}
\setcounter{footnote}{0}

A lemma that we prove in this appendix is essentially Lemma 3.2.1 of \cite{Jn-temp}. It covers a case not covered by  
Lemma 7.1 of \cite{T-temp}. The proof that we include here  uses methods (and notation) of the proof of  Lemma 7.1 of \cite{T-temp}. We need this simple result to know that the last step of the construction of elements in packets described in the previous section is compatible with \cite{Moe-T}. The claim is about partially defined functions. We do not recall their definition here (one can find it in  \cite{Moe-T}).

\begin{lemma}
\label{dodatak}
Let $\pi$ be an irreducible square integrable representation of a classical group.
Suppose $\d(\rho,2)\in \text{Jord}(\pi)$ and $\epsilon_\pi(\d(\rho,2))=1$. Then  there   exists an irreducible representation $\pi'$ of a classical 
group of the same series, such that 
\begin{equation}
\label{emb}
\pi \h \text{Ind}( |\text{det}|_F^{1/2} \rho\o\pi').
\end{equation}
 Further, any such $\pi'$ is square integrable and  $\text{Jord}(\pi')=\text{Jord}(\pi)\backslash \{\d(\rho,2)\}$.
The representation $\text{Ind}( |\text{det}|_F^{1/2} \rho\o\pi')$ has a unique irreducible subrepresentation and one gets  $\epsilon_{\pi'}$ from $\epsilon_{\pi}$ by restriction. 
\end{lemma}

\begin{proof} The definition of $\epsilon_\pi(\d(\rho,2))=1$  implies that there   exists an embedding of type \eqref{emb}.
 Further, $\pi'$ is square integrable by Remark 3.2 of \cite{Moe-Ex}. Now  $\text{Jord}(\pi')=\text{Jord}(\pi)\backslash \{\d(\rho,2)\}$ by (i) in Proposition 2.1 of \cite{Moe-T}.

The fact that the representation $\text{Ind}( |\text{det}|_F^{1/2} \rho\o\pi')$ has a unique irreducible subrepresentation follows directly applying the structure formula of \cite{T-Str} (Theorems 5.4 and 6.4 there), using $\d(\rho,2)\not\in \text{Jord}(\pi')$ and Lemma 3.6 of \cite{Moe-T} (one shows that $ |\text{det}|_F^{1/2} \rho\o\pi'$ has multiplicity one in the Jacquet module of $\text{Ind}( |\text{det}|_F^{1/2} \rho\o\pi')$, and then apply Frobenius reciprocity).

The proof that one gets  $\epsilon_{\pi'}$ by restricting  $\epsilon_\pi$  can be more or less extracted from the proof of Lemma 8.1 of \cite{T-temp}. We   explain this below. 

Let $\d(\rho',c)\in \text{Jord}(\pi')$ and suppose that $c_-$ is defined. Then (C) of the proof of Lemma 8.1 of \cite{T-temp} proves also that $\epsilon_{\pi}(\d(\rho',c)\d(\rho',c_-)) = \epsilon_{\pi'}(\d(\rho',c)\d(\rho',c_-))$ (one needs to take $a=2$ there).

Let $\d(\rho',c)\in \text{Jord}(\pi')$. Suppose that $c$ is odd and that $\epsilon_{\pi'}(\d(\rho',c))$ is defined (then $\rho\not\cong\rho'$). Let $b$ be the maximal element among such $c$'s for this $\rho'$. Now (F) of the proof of Lemma 8.1 of \cite{T-temp} also proves  that $\epsilon_{\pi}(\d(\rho',b)) = \epsilon_{\pi'}(\d(\rho',b))$ (again, one takes  $a=2$ there).

Let $\d(\rho',c)\in \text{Jord}(\pi')$ and suppose that $c$ is even. Let $b$ be the minimal element among such $c$'s (for this $\rho'$). 

First, consider  the case $\rho\not\cong \rho'$.
Then (A) of the proof of Lemma 8.1 of \cite{T-temp} also proves  that $\epsilon_{\pi}(\d(\rho',b)) = \epsilon_{\pi'}(\d(\rho',b))$ (one takes  $a=2$ there).

It remains to consider the case $\rho\cong \rho'$. Suppose $\epsilon_{\pi'}(\d(\rho,b)) =1$. Then 
$$
\pi'\h \text{Ind}(\d([|\text{det}|_F^{1/2}\rho,|\text{det}|_F^{(b-1)/2}\rho])\o \tau)
$$
 for some $\tau$, which implies 
$$
\pi\h \text{Ind}(\d([|\text{det}|_F^{1/2}\rho,|\text{det}|_F^{(b-1)/2}\rho])\o|\text{det}|_F^{1/2}\rho\o \tau).
$$
This implies $\epsilon_{\pi}(\d(\rho,b)) =1$.

Suppose now $\epsilon_{\pi}(\d(\rho,b)) =1$. Then $\epsilon_{\pi}(\d(\rho,b)\d((\rho,2)) =1$. This implies that
$$
\pi\h \text{Ind}(\d([|\text{det}|_F^{-1/2}\rho,|\text{det}|_F^{(b-1)/2}\rho])\o \s)
$$
for some irreducible square integrable representation $\s$. This and \eqref{emb} imply that 
$$
\d([|\text{det}|_F^{-1/2}\rho,|\text{det}|_F^{(b-1)/2}\rho])\o \s
$$
is in the Jacquet module of $ \text{Ind}( |\text{det}|_F^{1/2} \rho\o\pi')$. Now, a simple analysis based on the structure obtained in \cite{T-Str} implies that a subquotient of the form $\d([|\text{det}|_F^{-1/2}\rho,|\text{det}|_F^{(b-1)/2}\rho])\o *$ is in the Jacquet module of $\pi'$. Using this, section 7 of \cite{T-temp} implies $\epsilon_{\pi'}(\d(\rho,b)) =1$. This completes the proof.
\end{proof}

\end{document}